%% file: strahlen23.tex
\documentclass[12pt,a4paper]{amsart}
\usepackage{graphicx,subfigure}
\usepackage{diagrams}
\include{symboldefs}
\include{standard}

\numberwithin{equation}{section}

\newcommand{\hide}[1]{}

\newcommand{\Csep}[2]{C_{#1}^{#2}}
\newcommand{\Csepp}[2]{\dot C_{#1}^{#2}}
\newcommand{\Cseppp}[2]{\ddot C_{#1}^{#2}}

\newcommand{\sep}[1]{\dot #1}
\newcommand{\sepp}[1]{\ddot #1}

\newcommand{\norm}{\operatorname{n}}
\newcommand{\Blognorm}{\Blog^{\norm}}

\hide{
\newcommand{\Csep}[2]{C_{#1}^{#2}} 
\newcommand{\Csepp}[2]{\hat C_{#1}^{#2}} 
\newcommand{\Cseppp}[2]{\hat{\hat C}_{#1}^{#2}}
}

\newcommand{\Wiggle}{\mathcal{W}}

\newlength{\captionwidth}
\title[Dynamic rays of entire functions]{Dynamic rays of bounded-type 
           entire functions}
\date{\today}

\author[G.\ Rottenfu{\ss}er]{G\"unter Rottenfu{\ss}er}
\address{G\"orresstra{\ss}e 20, 80798 M\"unchen, Germany}
\email{guenter.rottenfusser@gmx.net}

\author[J.\ R\"uckert]{Johannes R\"uckert}
\address{School of Engineering and Science, Jacobs University Bremen (formerly International University Bremen), P.O.~Box 750 561, 
28725 Bremen, Germany}
\email{jrueckert@world.iu-bremen.de}

\author[L.\ Rempe]{Lasse Rempe}
\address{Department for Mathematical Sciences, University of Liverpool, Liverpool L69 7ZL,
United Kingdom}
\email{l.rempe@liv.ac.uk}

\author[D.\ Schleicher]{Dierk Schleicher}
\address{School of Engineering and Science, Jacobs University Bremen, P.O.~Box 750 561, 
28725 Bremen, Germany}
\email{dierk@jacobs-university.de}
\thanks{
Several of us were supported by a Doktorandenstipendium (G.R., J.R.) or a postdoctoral fellowship (L.R.) of the German Academic Exchange Service DAAD, and by a grant from the Fields institute in Toronto (J.R., L.R., D.S.). L.R.\ was
 partially supported by an EPSRC grant at the University of Warwick,
 and later by EPSRC Fellowship EP/E052851/1.
The authors gratefully acknowledge support by the European networks CODY and 
 HCAA, and their predecessors.}

\begin{document}

\begin{abstract} 
We construct an entire function in the Eremenko-Lyubich class $\B$
 whose Julia set has only bounded path-components. 
 This answers a question of Eremenko from 1989 in the negative.
   
 On the other hand, we show that for many functions in $\B$, 
  in particular those of finite order, every
  escaping point can be connected to $\infty$ by a curve of escaping
  points. This gives a partial positive answer to the aforementioned
  question of Eremenko, and answers a question of Fatou from 1926.
 \end{abstract}

\maketitle

 \section{Introduction}

The dynamical study of transcendental 
entire functions was initiated by
  Fatou in 1926 \cite{fatou}. 
  As well as being a fascinating
  field in its own right, the topic has
  recently received
  increasing
  interest partly
  because transcendental phenomena
  seem to be deeply linked with the
  behavior of polynomials in
  cases where the degree gets
  large. A recent example is provided by
  the surprising results of
  Avila and Lyubich \cite{avilalyubichfeigenbaum2}, who showed
  that a constant-type Feigenbaum quadratic polynomial whose Julia
  set has positive measure would have hyperbolic dimension less than
  two. This phenomenon occurs naturally in transcendental dynamics, 
  see \cite{urbanskizdunik1}.
  Other interesting applications of
  transcendental dynamics
  include the study
  of the standard family of circle maps
  and the use of Newton's method
  to study zeros of transcendental
  functions.

In his seminal
  1926 article, Fatou observed
  that
  the Julia sets of certain explicit entire 
  functions, such as
  $z\mapsto r\sin(z)$, $r\in\R$, contain
  curves of points that escape to 
  infinity under iteration. He then 
  remarks
\begin{quotation}
Il serait int\'eressant de rechercher si cette propri\'et\'e n'appartiendrait pas \`a des substitutions beaucoup plus g\'en\'erales.
\footnote{ ``It would be interesting to investigate
              whether this property might not hold
              for much more general functions.''}
\end{quotation}

Sixty years later,
  Eremenko \cite{alexescaping}
  was the first to
  undertake a thorough
  study of the
   \emph{escaping set} 
\[
 I(f) := \{z\in\C\,:\, |f^{\circ n}(z)|\to \infty \} 
\]
  of an arbitrary
  entire transcendental function. 
  In particular, he showed that
  every component of $\cl{I(f)}$ is 
  unbounded, and asks whether in fact
  each component of $I(f)$ is 
  unbounded. We will call this problem (the weak form of)
  \emph{Eremenko's conjecture}. 
  He also states that
  \begin{quotation}
   It is plausible that the set $I(f)$ always has the following property:
    every point $z\in I(f)$ can be joined with $\infty$ by a curve in
    $I(f)$. 
  \end{quotation}
 This can be seen as making Fatou's original question 
   more precise, and
  will be referred to in the following as 
  the 
  \emph{strong form} of Eremenko's 
  conjecture.

 These problems are of particular 
  importance since the existence of such
   curves can be used to study 
   entire functions using
   combinatorial methods. This is
   analogous to the notion of 
  ``dynamic rays'' 
   of polynomials 
   introduced by Douady 
   and Hubbard \cite{orsay}, which
   has proved to be one of the 
   fundamental tools for the
   successful study of polynomial 
   dynamics. Consequently, Fatou's
   and Eremenko's questions are among
   the most prominent open problems
   in the field of transcendental
   dynamics. 

 We will show that, in general, the 
  answer to Fatou's question
  (and thus also to Eremenko's 
   conjecture in its strong form) is 
   negative,
  even when restricted to the {\em Eremenko-Lyubich class}
  $\B$ of entire functions with a
  bounded set of singular values. 
  (For such functions, all escaping 
  points lie in the Julia set. The class 
  $\B$ appears to be a very natural setting
  for this type of problem; compare also \cite{boettcher}.) 
 \begin{thm}[Entire Functions Without Dynamic Rays] \label{thm:counterexample}
  There exists a hyperbolic entire function $f\in\B$ 
   such that every path-connected
   component of $J(f)$ is bounded.
 \end{thm}
\begin{remark}
 In fact, it is even possible to ensure that $J(f)$ contains no 
  nontrivial curves at all (Theorem \ref{thm:pointcomponents}).
\end{remark}

 On the other hand, we show that the 
  strong form of Eremenko's conjecture
  does hold for a large subclass of $\B$.
  Recall that $f$ has 
   \emph{finite order} if 
   $\log \log |f(z)| = O(\log|z|)$ as 
   $|z|\to\infty$.

\begin{thm}[Entire Functions With Dynamic Rays] 
\label{thm:positive}
  Let $f\in\B$ be a function of finite 
   order, or more generally
   a finite composition of such functions. 
  Then every point $z\in I(f)$ can be
   connected to $\infty$ by a curve 
   $\gamma$ such that
   $f^{\circ n}|_{\gamma}\to\infty$ uniformly. 
 \end{thm}
\begin{remark}
Observe that while $\B$ is invariant under finite compositions, the property of
  having finite order is not.
\end{remark}

 \begin{figure}
  \begin{center}
   \includegraphics[width=.99\textwidth]{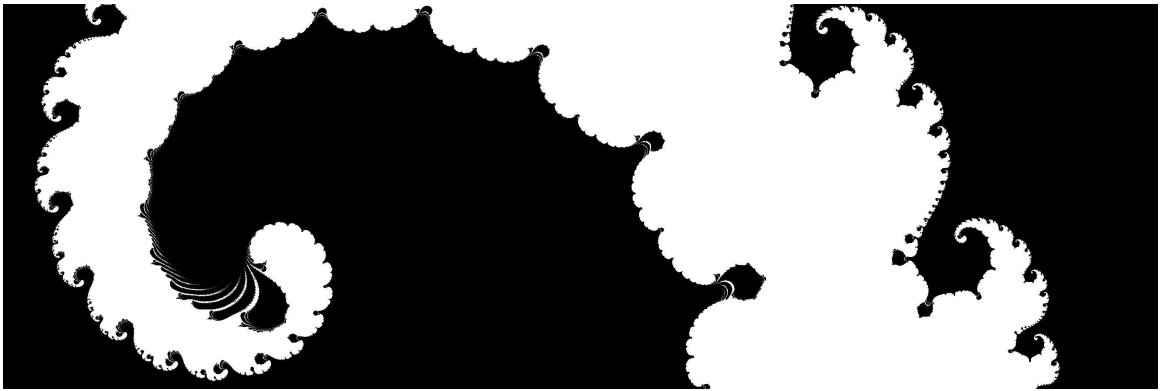}
  \end{center}
  \caption{Julia set for the Poincar\'e function around a
    repelling fixed point of a postcritically-finite quadratic polynomial. 
    This function belongs to $\B$ and has finite order; hence
    Theorem \ref{thm:positive} applies to it. (The Julia set, plotted in
    black, is nowhere dense, but some details
    are too fine to be visible; this results in the appearance of solid
    black regions in the figure.)
\label{fig:strahlen_spiral}}
 \end{figure}

Theorem \ref{thm:positive} applies to a large and natural class of functions, 
 extending considerably
 beyond those that were previously studied. 
 As an example, suppose that $p$ is a polynomial with connected Julia set,
 and let $\alpha$ be a repelling fixed point of $p$. By a classical theorem of
 K{\oe}nigs, $p$ is conformally conjugate to the linear map 
 $z\mapsto \mu z$ (where $\mu=p'(\alpha)$ is the multiplier of $\alpha$). 
 The inverse of this conjugacy extends to an entire function $\psi:\C\to\C$ 
 with $\psi(0)=\alpha$ and $\psi(\mu z) = f(\psi(z))$.
 This map 
 $\psi$, called a \emph{Poincar\'e function}, has finite order. The 
 set of singular values of $\psi$ is the postcritical set of $p$, which is
 bounded since $J(p)$ is connected. So $\psi\in\B$. 
 The properties of $\psi$ are 
 generally quite
 different from e.g.\ those of the commonly considered exponential and
 trigonometric functions; for example, if $\mu\notin\R$, then the 
 \emph{tracts} of $\psi$ will spiral near infinity. (Compare
 Figure \ref{fig:strahlen_spiral}.) 

 More generally, given any point
  $z\in J(p)$, we can find a sequence of rescalings of iterates of $p$
  that converges to an entire function $f:\C\to\C$ of finite order by the 
  ``Zalcman lemma'' (see e.g.\ \cite{zalcmannormal}). 
  Again, since $J(p)$ is connected, we have $f\in\B$. 
  Theorem
  \ref{thm:positive} implies that, for all such functions, and their
  finite compositions, the escaping set consists of rays. 

On the other hand, given $\eps>0$, 
  the counterexample from Theorem \ref{thm:counterexample}
 can be constructed such that
 $\log \log |f(z)| = O(|\log z|^{1+\eps})$
 (see Proposition \ref{prop:smallgrowthexample}), so Theorem \ref{thm:positive} is not far from
 being optimal. 

We remark that our methods are purely local in the sense that they apply
 to the dynamics of a --- not necessarily globally defined --- 
 function within any number of logarithmic singularities over $\infty$. 
 Roughly speaking, let $f$ be a function defined on a union of unbounded
  Jordan domains such that
  $f$ is a universal covering of $\{|z|>R\}$ on each of these domains,
  and such that only finitely many of the domains intersect
  any given compact set. 
  (In fact, our setting is even more general than this;
   see Section~\ref{sec:classB} for the class of functions we consider.)   
  We will provide sufficient conditions that ensure that
  every point $z\in I(f)$ eventually maps into 
  a curve in $I(f)$ ending at $\infty$. In particular, we show
  that these conditions are satisfied if $f$ has finite order of growth. 
  (Our treatment also allows us to discuss under 
   which conditions
   \emph{individual} escaping components, identified by their
   \emph{external addresses}, are curves to $\infty$.)

For meromorphic functions, we have the following corollary.
\begin{cor}[Meromorphic Functions With Logarithmic Singularities]
\label{cor:mero}
  Let $f:\C\to\Ch$ be a meromorphic function of finite order. 
  \begin{enumerate}
   \item Suppose that $f$ has only finitely many poles and the
     set of finite singular values of $f$ is bounded. Then every
     escaping point of $f$ can be connected to $\infty$ or to a pre-pole
     of $f$ by a curve consisting of escaping points.
   \item Suppose that 
     $f$ has a logarithmic singularity over $\infty$.
     Then $J(f)$ contains uncountably many
     curves to $\infty$ consisting of escaping points. 
  \end{enumerate}
 \end{cor}
\begin{remark}
 The second part of the corollary applies e.g.\ to the classical
  $\Gamma$-function, which
  has infinitely many poles (at the negative integers), but a logarithmic
  singularity to the right. 
\end{remark}

 We note finally that our results also apply to the setting of
  ``random iteration'' (see \cite{marksurvey}). For example,
  consider a sequence $\F=(f_0,f_1,f_2,\dots)$, where the $f_j$ are
  finite-order entire functions chosen from some given finite
  subset of $\B$. If we consider the 
  functions $\F_n := f_{n}\circ\dots\circ f_{0}$, and define
  $I(\F):=\{z\in\C:\F_n(z)\to\infty\}$, then every point of $I(\F)$ can be
  connected to infinity by a curve in $I(\F)$. 

 \subsection*{Previous results}
   In the early 1980s, Devaney  gave a complete
   description of the Julia set of any real exponential map that has an
   attracting fixed point; that is, 
   $z\mapsto\lambda\exp(z)$ with $\lambda\in(0,1/e)$
   (see \cite{devaneykrych}). This was the first
   entire function
   for which it was discovered that
   the escaping set (and in fact the Julia set) consists of curves
   to $\infty$. Devaney, Goldberg and Hubbard
   \cite{dgh} proved the existence of certain curves to $\infty$ in
   $I(f)$ for arbitrary exponential maps $z\mapsto \lambda\exp(z)$ and
   championed the idea that these should be thought of as analogs of
   dynamic rays for polynomials. 
   Devaney and Tangerman \cite{devaneytangerman} generalized this result
   to a large subclass of $\B$, namely those functions
   whose \emph{tracts} (see Section \ref{sec:classB}) are
   similar to those of the exponential map. 
   (This includes virtually all functions in the Eremenko-Lyubich class that 
    one can explicitly write down.)
   It seems that it was partly these developments that led
   Eremenko to pose the above-mentioned questions in his 1989 paper.

More recently, it was shown 
  in \cite{expescaping} that every escaping point of
  every exponential map can be connected to $\infty$ by a curve consisting of
  escaping points. This was the first time that 
  a complete classification of all escaping points, 
  and with it a positive answer to both of Eremenko's questions, 
  was given for a complete parameter space
  of transcendental functions. This result
  was carried over to the cosine family $z\mapsto a\exp(z)+b\exp(-z)$ in
  \cite{guenterdierk}.

 After our proof of Theorem \ref{thm:positive} was first announced, Bara\'nski
  \cite{baranskihyperbolic} obtained a proof of this result for 
  hyperbolic finite-order
  functions $f\in\B$ whose Fatou set consists of a single basin of
  attraction. 
  (In fact, Bara\'nski shows that for these functions every component of
   the Julia set is a curve to $\infty$; compare
   Theorem \ref{thm:baranskitype}.) 
  Together with more recent results \cite{boettcher} on the 
  escaping dynamics of functions in the Eremenko-Lyubich class, this provides
  an alternative proof of our theorem when $f$ is of finite order.

  A very interesting and surprising case in which the weak form of
   Eremenko's conjecture can be resolved 
   was discovered by Rippon and Stallard
   \cite{ripponstallardfatoueremenko}. They showed that the escaping
   set of a function 
   with a multiply-connected wandering domain 
   consists of a single, unbounded, connected component. 
   (Such functions never belong to the Eremenko-Lyubich class $\B$.) 
   In fact, they showed that, for any transcendental entire function, the 
   subset $A(f)\subset I(f)$ of ``points escaping at the fastest
   possible rate'', as
   introduced by Bergweiler and Hinkkanen 
   \cite{walteraimo}, has only unbounded components.  Also, recently   
   \cite{eremenkoproperty} 
   the weak form of Eremenko's conjecture was established for 
   functions $f\in\B$ whose postsingular set is bounded (which applies,
   in particular, to the hyperbolic counterexample constructed in Theorem
   \ref{thm:counterexample}).

There has been substantial interest in the set $I(f)$ not only from the point 
 of its topological structure, but also because of its interesting properties 
 from the point of view of Hausdorff dimension. 
 The reasoning is often parallel, and progress on the topology of 
 $I(f)$ has entailed progress on the Hausdorff dimension. 
 For many functions $f$ the set $I(f)$ is an 
 uncountable union of curves, each of which is homeomorphic to either $\R^+$ 
 (a dynamic ray) or $\R^+_0$ (a dynamic ray that lands at an escaping point). 
 Often, the Hausdorff dimension of all the rays is $1$, while the endpoints 
 alone have Hausdorff dimension $2$, or even infinite planar Lebesgue 
 measure. 
 This ``dimension paradox'' was discovered by Karpi\'nska for hyperbolic 
 exponential maps \cite{Karpinska}, for which the topology of Julia sets was 
 known. 
 In later extensions for 
 arbitrary exponential maps \cite{expescaping} and for the cosine family 
 \cite{guenterdierk}, the new parts were the topological classifications; 
 while analogous results on the Hausdorff dimension followed from the methods 
 of Karpi\'nska and McMullen; see also \cite{dierkcosine,dierkmonthly} for 
 extreme 
 results where every point in the complex plane is either on a dynamic ray 
 (whose union still has dimension one) or a landing point of those rays -- 
 so the 
 landing points of this one-dimensional set of rays is the entire complex 
 plane with only a one-dimensional set of exceptions. Recently, it was
 shown by Bara\'nski \cite{baranskihausdorff} that the dimension paradox
 also occurs for finite-order entire transcendental functions that are
 hyperbolic with a single basin of attraction. In fact,
 the Hausdorff dimension of $I(f)$ 
 is two for any entire function $f\in\B$ of finite order, which follows
 from Bara\'nski's result by \cite{boettcher} and was also shown
 directly and independently by Schubert. This generalizes
 McMullen's results \cite{hausdorffmcmullen}
 on the escaping sets of exponential and trigonometric functions. 
 Further studies of the Hausdorff dimension of the escaping set for
 $f\in\B$ can be found in \cite{bergweilerkarpinskastallard,escapingdim}.

\subsection*{Structure of the article}
In Section \ref{sec:classB}, we define logarithmic coordinates, in which we will perform most of our constructions. Some properties of functions in logarithmic coordinates are proved in Section \ref{sec:general}. In Section \ref{sec:head}, we show that the escaping set of a function in logarithmic coordinates consists of arcs if the escaping points can be ordered according to their ``speed'' of escape. We call this property the \emph{head-start condition}. Classes of functions that satisfy this condition, in particular logarithmic transforms of finite order functions, are discussed in Section \ref{sec:growth}.

In Section \ref{sec:counter}, we construct a function in logarithmic coordinates whose escaping set has only bounded path-components and in Section \ref{sec:realization}, we show how to translate this result into a bounded-type entire function. In an appendix, we recall several facts from hyperbolic geometry,
 geometric function theory and continuum theory.

\subsection*{Acknowledgements}
Most of all, we would like to thank Walter Bergweiler and Alex Eremenko for many
interesting discussions, and especially for introducing us to the method
 of using Cauchy integrals (Section \ref{sec:realization}) to construct
 entire functions, which we used in the proof of Theorem 
 \ref{thm:counterexample}. In particular, we would like to thank Professor 
 Eremenko for introducing us to the article \cite{eremenkogoldberg}. 
 We would also like to thank Adam Epstein, Helena Mihaljevic-Brandt, 
 Mikhail Lyubich, 
 Gwyneth Stallard, Phil Rippon, Norbert Steinmetz
 and Sebastian van Strien for many stimulating
 discussions on this work. We are indebted to the referee for many
 helpful suggestions. 

\subsection*{Notation and basic definitions}
  Throughout this article, we denote
   the Riemann sphere by $\Ch=\C\cup\{\infty\}$ and the right half plane by
   $\H:=\{z\in\C\,:\,\re z>0\}$. Also, we write
\[
	B_r(z_0) := \{z\in\C\,:\,|z-z_0|<r\} \quad\text{and}\quad
      	\H_R := \{z\in\C\,:\,\re z > R\}\;. 
\]
If $A\subset\C$, the closures of $A$ in $\C$ and $\Ch$ 
are denoted $\cl{A}$ and $\wh{A}$, respectively. 

Euclidean length and distance are denoted
$\ell$ and $\dist$, respectively.
If a domain $V\subset\C$ omits at least two points of the plane,
we similarly denote hyperbolic length and distance in $V$ by
$\ell_V$ and $\dist_V$. We shall often use the following well-known
 fact \cite[Corollary A.8]{jackdynamicsthird}: 
 if $V\subset\C$ is a simply connected domain, then the density
 $\lambda_V$ of the hyperbolic metric on $V$ satisfies
   \begin{equation} \label{eqn:standardestimate}
     \frac{1}{2\dist(z,\partial V)} \leq  
       \lambda_V(z) \leq \frac{2}{\dist(z,\partial V)}
   \end{equation}
  for all $z\in V$; we shall refer to this as the \emph{standard estimate}.

Let $f:\C\to\Ch$ be a meromorphic function. We call a point $a\in\Ch$ a 
 {\em singular value} of $f$ if for every open neighborhood $U$ of $a$, 
 there exists a component $V$ of $f^{-1}(U)$ such that 
 $f:V\to U$ is not bijective. 
 Denote the set of all finite singular values of $f$
 by $S(f)\subset\C$. 
 Clearly every critical value of $f$ belongs to $S(f)$.

Recall that $a\in\Ch$ is an \emph{asymptotic value} of $f$ if there
 exists a curve $\gamma:[0,\infty)\to\C$ with 
 $\lim_{t\to\infty}|\gamma(t)|=\infty$ such that $a=\lim_{t\to\infty} f(\gamma(t))$.
 (An example is given by $f(z)=\exp(z)$, $\gamma(t)=-t$, $a=0$.)
 Every asymptotic value of $f$ is a singular value; conversely,
 $S(f)$ is the closure of the set $\sing(f^{-1})$ of all
 finite critical and asymptotic values.

Let $f:\C\to\Ch$ be a transcendental entire or meromorphic function,
 and let $a\in\Ch$. Suppose there is some simply-connected open
 neighborhood $D\subset\C$ of $a$ and a component $U$ of 
 $f^{-1}(D\setminus\{a\})$ such that $f:U\to D\setminus\{a\}$ is a universal
 covering map. Then we say that $f$ has a \emph{logarithmic singularity}
 over $a$. In this case, $a$ is necessarily an asymptotic value of $f$;
 conversely $f$ has a logarithmic singularity over every isolated asymptotic
 value of $f$. For a further discussion of types of asymptotic values,
 see \cite{walteralexsingularities}. 

 As stated in the introduction, 
  we say that a transcendental entire function
  $f$ belongs to the {\em Eremenko-Lyubich class} $\B$ if 
  $S(f)$ is bounded.
 By $J(f)$ we denote the {\em Julia set} of $f$, 
  i.e.\ the set of points at which the sequence of functions 
  $\{f,f\circ f, \dots, f^{\circ n},\dots\}$ does not form a normal family in 
  the sense of Montel. The reader is referred to \cite{jackdynamicsthird} for a 
  general introduction
  to the dynamics in one complex variable, and to 
  \cite{waltermero,waltervorlesung,dierkentire} for background on transcendental
  dynamics.

We conclude any proof by the symbol $\proofsymbol$. 
 The proofs of separate claims within an argument are concluded by
 $\subproofsymbol$.

\section{The Eremenko-Lyubich Class \texorpdfstring{$\B$}{B} %
         and the Class \texorpdfstring{$\Blog$}{B\_log}} 
\label{sec:classB}

\subsection*{Tracts}
Let $f\in\B$, and let $D\subset\C$ be a bounded Jordan domain that
 contains $S(f)\cup\{0,f(0)\}$. 
 Setting $W := \C\setminus\cl{D}$, it is easy to see that
 every component $V$ of
\[
	 \V := f^{-1}(W) 
\]
is an unbounded Jordan domain, and that $f:V\to W$ is a universal
   covering. (In other words, $f$ has only logarithmic singularities
   over $\infty$.) The components of $\V$ are called the 
   \emph{tracts} of $f$. Observe that a given
   compact set $K\subset\C$ will intersect 
   at most finitely many tracts of $f$.

 \subsection*{Logarithmic coordinates}
  To study logarithmic singularities, it is natural to apply a
   logarithmic change of coordinates (compare \cite[Section 2]{alexmisha}).
   More precisely, let 
   $\T := \exp^{-1}(\V)$ and $H := \exp^{-1}(W)$. 
   Then there is a continuous
    function $F:\T\to H$ (called a 
    \emph{logarithmic transform} of $f$) such that the following
    diagram commutes:
  \begin{diagram}
 	\T& &\rTo^{F} & & H \\
       \dTo^{\exp}& & & &\dTo_{\exp} \\
    \V  &  &\rTo_f & &W .	
  \end{diagram}
The components of $\T$ are called the \emph{tracts} of $F$. 
 Note that $F$ is unique up to postcomposition by a map of the
 form $z\mapsto z + k(z)$, where $k:\T\to 2\pi i \Z$ is continuous
 (and hence constant on every tract of $F$). 

 By construction, the function $F$ and its domain
  $\T$ have the following properties, see also Figure \ref{Fig:Tracts}:
  \begin{enumerate}
   \item \label{item:halfplane} $H$ is a $2\pi i$-periodic Jordan domain
     that contains a right half plane;
   \item \label{prop:a}
	every component $T$ of $\T$ is an unbounded Jordan domain with real 
        parts bounded below, but unbounded from above;
   \item  The components of $\cl{\T}$ have disjoint
     closures and accumulate only at infinity; that is,
     if $z_n\in\T$ is a sequence of points all belonging to different tracts,
     then $z_n\to\infty$;
     \label{item:disjointunion}
   \item for every component $T$ of $\T$, $F:T\to H$ 
     is a conformal isomorphism. In particular, $F$ extends continuously to 
     the closure
     $\cl{\T}$ of $\T$ in $\C$; \label{item:extensiontoclosure}
   \item for every component $T$ of $\T$, $\exp|_{T}$ is injective;
     \label{item:expinjective}
   \item $\T$ is invariant under translation by $2\pi i$. \label{item:twopiinvariant}
  \end{enumerate}

\begin{figure}[htb]
\begin{center}
\setlength{\unitlength}{1cm}
\begin{picture}(10,5.3)
\put(1.1,4.3){$H$}
\put(0.7,0.2){$0$}
\put(0.35,2.2){$2\pi i$}
\includegraphics[viewport=000 000 555 291,clip,width=10cm]{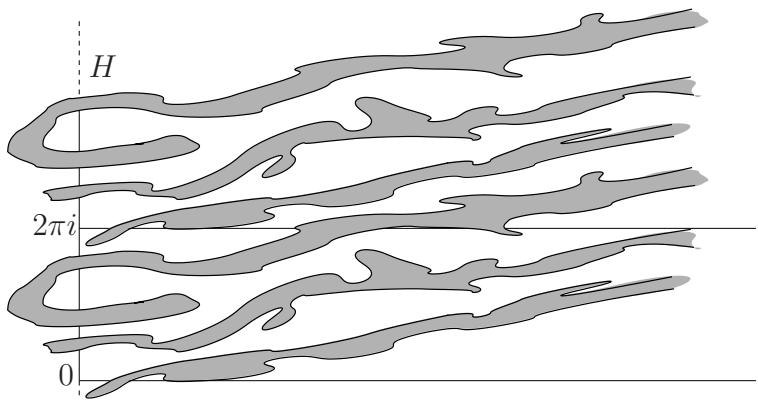}
\end{picture}
\caption{\label{Fig:Tracts} The domain of definition for a function $F\in\Blog$ is $2\pi i$-periodic. Note that a tract $T$ need not be contained in $H$.} 
\end{center}
\end{figure}

  We will denote by $\Blog$ the class of all 
   $F:\T\to H$ such that $H$, $\T$ and $F$ satisfy 
  (\ref{item:halfplane}) to (\ref{item:twopiinvariant})  
   regardless of whether they arise from an entire function
   $f\in\B$ or not. In particular, if $f:\C\to\Ch$ is any meromorphic
   function that has one or several
   logarithmic singularities over $\infty$, then 
   we can associate to $f$ a function $F\in\Blog$ that encodes
   the behavior of $f$ near its logarithmic singularities.

 If $F\in \Blog$ and $T$ is a tract of $F$, we denote the
  inverse of the conformal isomorphism $F:T\to H$ by $\FInv{T}$. 

 \subsection*{Normalized and disjoint-type functions}
   Let $F:\T\to H$ be a function of class $\Blog$.
   A simple application of Koebe's $1/4$-theorem shows that
   there is $R_0>0$ such that
     \begin{equation}
      |F'(z)| \geq 2   \label{eqn:expansion}
     \end{equation}
   when $\re F(z)\geq R_0$; see
    \cite[Lemma 2.1]{alexmisha}.
  In the following, we refer to the property (\ref{eqn:expansion})
   as {\em expansivity} of $F$. 

  We shall say that $F$ is \emph{normalized} if
   $H$ is the right half plane $\H$ and furthermore
   (\ref{eqn:expansion}) holds for all $z\in\T$. We denote the
   set of all such functions by $\Blognorm$.

  Note that we can pass from any function $F\in\Blog$ to a normalized one
   by restricting $F$ to $\T' := F^{-1}(\H_{R_0})$ (where $R_0$
   is as above) and applying the change of variable $w=z-R_0$.
   For this reason, it is usually no loss of generality to assume that
   $F\in\Blognorm$. 
 
  Let us also say that $F$ is of \emph{disjoint type} if
   $\cl{\T}\subset H$. It is easy to see that
   a function $f\in\B$ has a disjoint-type logarithmic
   transform if and only if $S(f)$ is contained in the immediate basin
   of an attracting fixed point of $f$. (This is the setting considered
   by Bara\'nski \cite{baranskihyperbolic}.) Note that we might not
   be able to normalize such a function in the above-mentioned 
   manner without losing the disjoint-type property. However,
   if $F$ is of disjoint type, then $F$ will be uniformly expanding
   with respect to the
   hyperbolic metric on $H$. 
 \begin{lem}[Uniform Expansion for Disjoint-Type Maps]
   \label{lem:hyperbolicexpansion}
  Suppose $F:\T\to H$ is of disjoint type. 

  Then there exists a constant
   $\Lambda>1$ such that the derivative of $F$ with respect to the
   hyperbolic metric on $H$ satisfies
     $\|DF(z)\|_H = \lambda_{\T}(z) / \lambda_H(z) \geq \Lambda$
   for all $z\in \T$. 

   In particular, 
     $\dist_H(F(z),F(w)) \geq \Lambda \dist_H(z,w)$
  whenever $z$ and $w$ belong to the
   same tract of $F$.
 \end{lem}
 \begin{proof}
   Equality in the first claim is satisfied because $F:T\to H$ is
    a conformal isomorphism for every component $T$ of $\T$. By Pick's theorem
    \cite[Theorem 2.11]{jackdynamicsthird}, we have $\|DF(z)\|_H>1$ for all
    $z\in V$, and we have $\lambda_{\T}(z)/\lambda_H(z)\to\infty$ as
    $z$ tends to the boundary of $V$ (in $\C$). Since $\T$ is $2\pi i$-periodic,
    it remains to
    show that $\liminf_{\re z\to+\infty} \lambda_{\T}(z)/\lambda_H(z)>1$. By the
    standard estimate (\ref{eqn:standardestimate}), we have
    $\lambda_{\T}(z)\geq 2\pi$ for all $z\in\T$, while $\lambda_H(z)\to\infty$
    as $\re z\to +\infty$. This proves the first claim.

  The second claim follows from the first: if $\gamma$ is the hyperbolic 
   geodesic of $H$ that connects $F(z)$ and $F(w)$, then 
   $F_T^{-1}(\gamma)$ has length at most $\Lambda \dist_H(z,w)$.    
 \end{proof}

 \subsection*{Combinatorics in \texorpdfstring{$\Blog$}{B\_log}}

  Let $F\in \Blog$; we denote the Julia set and the
   set of escaping points of $F$ by
    \begin{align*}
      J(F) &:= \{z\in\cl{\T}: F^{\circ n}(z) \text{ is defined and in $\cl{\T}$
                   for all $n\geq 0$}\}
                \quad\text{and} \\
      I(F) &:= \{z\in J(F): \re F^{\circ n}(z)\to \infty\}\;.
    \end{align*}
  If $f\in\B$ and $F$ is a logarithmic transform of $f$, then clearly
   $\exp(I(F))\subset I(f)$. Furthermore, if $F$ is normalized or
   of disjoint type, then
   $\exp(J(F))\subset J(f)$, respectively $\exp(J(F))= J(f)$. 
   (See \cite[Lemma 2.3]{boettcher}.) 
   For $K>0$ we also define more generally
   \[ J^K(F) := \{z\in J(F):
                   \re F^{\circ n}(z)\geq K \text{ for all $n\geq 1$}\}\;. 
      \]
The partition of the domain of $F$ into tracts suggests a 
   natural way to assign symbolic dynamics to points in $J(F)$.
   More precisely, let $z\in J(F)$
    and, for $j\geq 0$, let
    $T_j$ be the tract of $F$ with $F^{\circ j}(z)\in \cl{T}_j$. Then the sequence
    \[ \s := T_0 T_1 T_2 \dots \]
    is called the \emph{external address} of $z$. More generally,
    we refer to any sequence of tracts of $F$ as an external address
    (of $F$). If $\s$ is such an external address, we define the closed set
    \[ J_{\s} := \{z\in J(F): z\text{ has address $\s$}\}\;; \]
    we define $I_{\s}$ and $J^K_{\s}$ in a similar fashion (note that 
     $J_{\s}$, and hence $I_{\s}$ and $J^K_{\s}$,
      may well be empty for some addresses). 

  We denote the one-sided {\em shift operator} on external 
   addresses by $\sigma$. In other words, $\sigma(T_0 T_1 T_2\dots)=T_1 T_2\dots$.
 
\begin{defn}[Dynamic Rays, Ray Tails]
\label{Def:RayTails}
Let $F\in\Blog$. A {\em ray tail} of $F$ is an injective curve
\[
	\gamma: [0,\infty)\to I(F)
\]
  such that $\lim_{t\to\infty}\re F^{\circ n}(\gamma(t))=+\infty$ for all
  $n\geq 0$ and such that
  $\re F^{\circ n}(\gamma(t))\to+\infty$ uniformly in $t$ as $n\to\infty$. 

Likewise, we can define ray tails for an entire function $f$. A
 \emph{dynamic ray} of $f$ is then a maximal injective curve 
 $\gamma: (0,\infty)\to I(f)$ such that 
 $\gamma|_{[t,\infty)}$ is a ray tail for every $t>0$.
\end{defn}
\begin{remark}
 If $z$ is on a ray tail, then $z$ is either on 
  a dynamic ray or the landing point of such a ray. 
  In particular, it is possible under our terminology for 
  a ray tail to
  properly contain a dynamic ray.
\end{remark}

In Sections \ref{sec:head} and \ref{sec:growth}, we will construct ray tails for
 certain functions in class $\Blog$, and in particular for logarithmic
 transforms of the functions treated in Theorem
 \ref{thm:positive} and Corollary \ref{cor:mero}. By the following fact,
 this will be sufficient to complete our objective.

 \begin{prop}[Escaping Points on Rays] \label{prop:classification}
  Let $f:\C\to\C$ be an entire function and let
   $z\in I(f)$. Suppose that
  some iterate $f^{\circ k}(z)$ is on a ray tail $\gamma_k$ of $f$. 
  Then either $z$ is on a ray tail, or 
  there is some $n\leq k$ such that $f^{\circ n}(z)$ belongs to
       a ray tail that contains an asymptotic value of $f$.

  In particular, there is a curve $\gamma_0$ connecting $z$ to $\infty$
   such that $f^{\circ j}|_{\gamma_0}$ tends to $\infty$ uniformly
   (in fact, $f^{\circ k}(\gamma_0)\subset \gamma_k$).
 \end{prop}
 \begin{proof}
   Let the ray tail be parametrized as
    $\gamma_k:[0,\infty)\to\C$. Let 
    $\gamma_{k-1}:[0,T)\to\C$ be a maximal lift of $\gamma_k$ starting
    at $f^{\circ(k-1)}(z)$. That is, 
    $\gamma_{k-1}(0)=f^{\circ (k-1)}(z)$, 
    $f(\gamma_{k-1}(t))=\gamma_k(t)$ and there is no extension of
    $\gamma_{k-1}$ to a larger interval that has these properties.
    (Such a maximal lift exists e.g.\ by Zorn's lemma.)

 If $T=\infty$, then clearly
    $\gamma_{k-1}(t)\to\infty$ as $t\to\infty$. 
    Otherwise, $w=\lim_{t\to T}\gamma_{k-1}(t)$ exists
    in $\Ch$. If $w\neq\infty$, we could extend $\gamma_{k-1}$
    further (choosing any one of the possible branches
    of $f^{-1}$ in the case where $w$ is a critical point), contradicting 
    maximality of $T$. Thus $w=\infty$ and, in particular, $\gamma_k(T)$ is
    an asymptotic value of $f$.

    In either case, we have found a curve $\gamma_{k-1}\subset
     f^{-1}(\gamma_k)\subset I(f)$ connecting $f^{\circ (k-1)}(z)$ to
     infinity. This curve is a ray tail, except possibly if
     $\gamma_k$ contained an asymptotic value of $f$.
     Continuing this method inductively, we are done. 
\end{proof}

\section{General Properties of Class \texorpdfstring{$\Blog$}{B\_log}}
\label{sec:general}

  In this section, we prove some general results for functions in 
   class $\Blog$. The first of these strengthens the aforementioned
   expansion estimate of
   \cite[Lemma 2.1]{alexmisha} by showing that such a function
   expands distances like an exponential map. 
   
\begin{lem}[Exponential Separation of Orbits] 
\label{lem:expansion}
Let $F\in\Blognorm$ and let $T$ be a tract of $F$. If $\omega,\zeta\in T$ are such that $|\omega-\zeta|\geq 2$, then 
     \[ 
	|F(\omega) - F(\zeta)| \geq 
         \exp(|\omega - \zeta|/8\pi)
         \cdot\min\{\re F(\omega),
             \re F(\zeta)\}\;. 
\]
\end{lem}
\begin{proof} Suppose without loss of generality that 
 $\re F(\omega) \geq \re F(\zeta)$. 
 Since $T$ has height at most $2\pi$, it follows by 
 the standard estimate (\ref{eqn:standardestimate})
 on hyperbolic distances that
\[
	|\omega-\zeta|/2\pi \leq \dist_T(\omega,\zeta)= \dist_{\H}(F(\omega), F(\zeta))\;.
\]
Let $\xi\in\H$ be a  point that satisfies $\dist_\H(F(\zeta),\xi)=\dist_{\H}(F(\omega), F(\zeta))$ and $\re F(\zeta)=\re \xi$, see Figure \ref{Fig:Hyperbolic}. We will estimate the Euclidean distance $s=|F(\zeta)-\xi|$. Then, $|F(\omega)-F(\zeta)|\geq s$. Let $\gamma$ be the curve consisting of three
 straight line segments pictured in Figure
  \ref{Fig:Hyperbolic}, connecting $F(\zeta)$ to $\xi$ 
  through $F(\zeta)+s$ and $\xi + s$. 
\begin{figure}[hbt]
\begin{center}
\setlength{\unitlength}{1cm}
\begin{picture}(6,5)
\put(1.5,0.5){$\xi$}
\put(2.86,0.5){$\xi+s$}
\put(1.44,1.6){$s$}
\put(1.4,2.7){$F(\zeta)$}
\put(2.8,2.7){$F(\zeta)+s$}
\put(3.4,1.6){$\gamma$}
\put(5.7,1){$S$}
\put(4.2,4){$F(\omega)$}
\put(0.8,4.5){$\H$}
\includegraphics[viewport=000 000 313 310,clip,width=6cm]{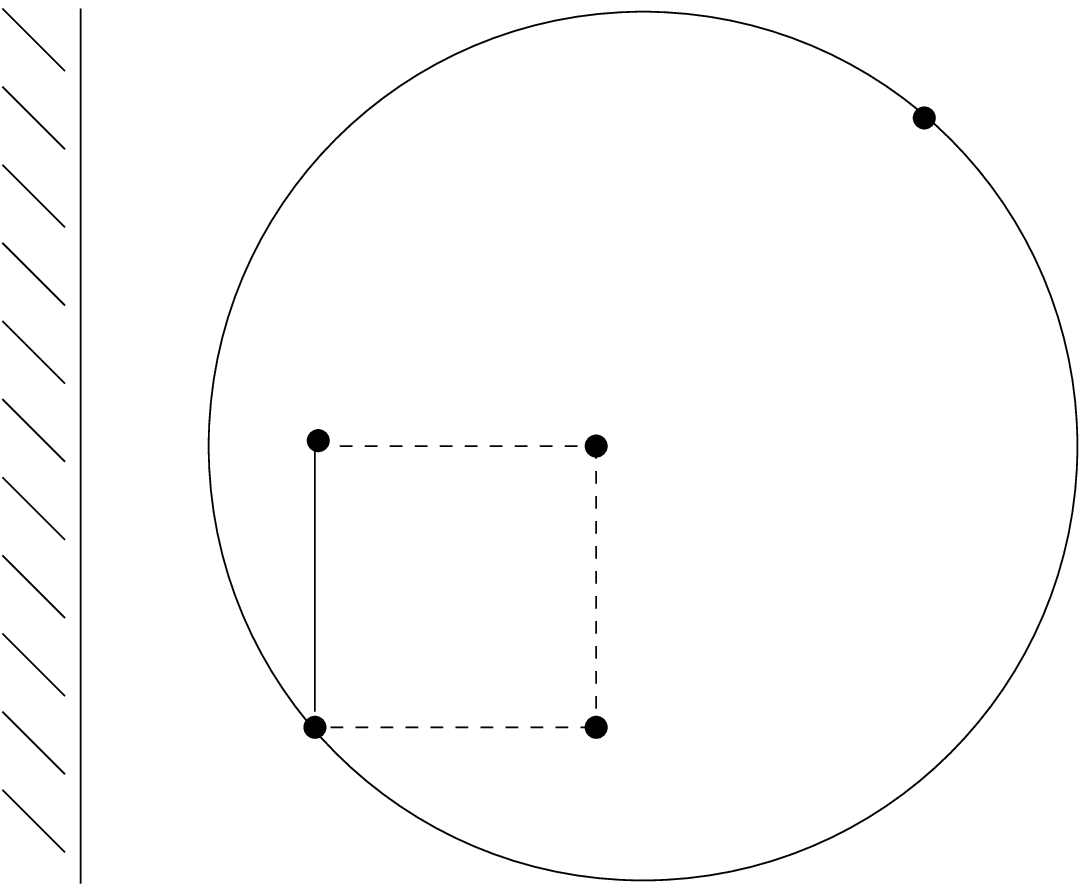}
\end{picture}
\end{center}
\caption{\label{Fig:Hyperbolic} The set $S$ of points of equal hyperbolic distance in $\H$ to $F(\zeta)$ is a Euclidean circle. Clearly, $\xi$ is the Euclidean closest point to $F(\zeta)$ on $S$ that satisfies $\re \xi \geq \re F(\zeta)$. We use the dotted line $\gamma$ to estimate $s$.} 
\end{figure}

It is easy to see that the hyperbolic length of the vertical part of $\gamma\subset\H$ is less than $1$. On the other hand, each of the horizontal
  parts of $\gamma$ has hyperbolic length precisely
 $\log\left( (\re F(\zeta)+s)/\re F(\zeta)\right)$. Hence, we get 
\[
	\frac{|\omega-\zeta|}{2\pi} < 2\log\left(\frac{\re F(\zeta) + s}{\re F(\zeta)}\right) + 1\;,
\]
and therefore
\[
	|F(\omega)-F(\zeta)|\geq s \geq \left(\exp\left(\frac{|\omega-\zeta|}{4\pi}-\frac{1}{2}\right)-1\right)\cdot\re F(\zeta)\;.
\]
Since  $e^{x-1/2}-1>e^{x/2}$ for $x\geq 2$, the claim follows.
\end{proof}
\begin{remark}
It follows from expansivity of $F$ that for any two distinct points $w,z$ with the same external address, there exists $k\in\N$ such that $|F^{\circ k}(w)-F^{\circ k}(z)|>2$. Hence Lemma \ref{lem:expansion} will apply eventually. 
\end{remark}

\begin{lem}[Growth of Real Parts]
\label{lem:realseparation}
Let $F\in\Blognorm$.
If $\zeta,\omega\in J(F)$ are distinct points with the same external address $\s$, then
\[ 
	\lim_{k\to\infty} \max(\re F^{\circ k}(\zeta),\re F^{\circ k}(\omega)) = \infty\;. 
\]
\end{lem}
\begin{proof} 
 Suppose that 
  $\zeta,\omega\in J(F)$ satisfy
  $\re F^{\circ k}(\zeta), \re F^{\circ k}(\omega) <S$ for some $S>0$ and
  infinitely 
  many $k\in\N$. For any tract $T$, the set 
  $\cl{T}\cap\{z\in\C\,:\,\re z \leq S\}$ is compact and thus has bounded 
  imaginary parts. Furthermore, up to translations in $2\pi i\Z$ there are only
  finitely many
  tracts of $F$ that intersect the line $\{\re z = S\}$ at all
  (this follows from property (\ref{item:disjointunion}) in the definition
  of $\Blog$).
  We conclude that there is $C>0$ such that 
    $|F^{\circ k}(\zeta) - F^{\circ k}(\omega)|<C$ whenever 
    $\re F^{\circ k}(\zeta), \re F^{\circ k}(\omega)<S$. In particular,
\[
|\zeta-\omega| \leq \frac{1}{2^k}\cdot |F^{\circ k}(\zeta)-F^{\circ k}(\omega)|
                \leq \frac{C}{2^k}
\]
  by expansivity of $F$ (we have $|(F^{-1}_{T})'(z)|<1/2$ for any tract
  $T$). Since this happens for infinitely many $k$, it follows that
   $\zeta=\omega$, as required.
\end{proof}

Note that Lemma \ref{lem:realseparation} does not imply that either $\zeta$ or 
 $\omega$ escapes: indeed, it is conceivable that both points have 
 unbounded orbits but return to some bounded real parts infinitely many times. 
 In the next section, we introduce a property,
 called a \emph{head-start condition}, which is
 designed precisely so that this does not occur.
 
As mentioned in the introduction, Rippon and Stallard \cite{ripponstallardfatoueremenko} showed that the escaping
set of every entire function $f$ contains unbounded connected sets. The following theorem is a
version of this result for functions in $\Blog$. 
\begin{thm}[Existence of Unbounded Continua in $J_{\s}$]
\label{thm:ripponstallard}
For every $F\in\Blog$  there exists $K\geq 0$
   with the following property. 
   If $z_0\in J^K(F)$ and $\s$ is the external address of $z_0$, then
   there 
   exists an unbounded closed connected set $A\subset J_{\s}$ 
   with
   $\dist(z_0,A)\leq 2\pi$.
\end{thm}
\begin{proof}
 We may assume without loss of generality that $F$ is normalized. 
  Choose $K$ large enough so that no bounded component of
   $\H\cap\cl{T}$ intersects the line 
   $\{\re z = K\}$ and set $z_k := F^{\circ k}(z_0)$.
   If $S\subset\C$ is an unbounded 
   set such that $S\sm B_{2\pi}(z_k)$ has 
   exactly one unbounded component, 
   let us denote this component by 
   $X_k(S)$.
  
  We claim that $X_k(\cl{T}_k)$ is non-empty and contained in $\H$ for all $k\geq 1$. 
   (However, this set is \emph{not} necessarily contained in $\H_K$.) 
   Indeed, this is trivial if $\cl{T}_k\subset \H$. Otherwise,
   let $\alpha^-$ and $\alpha^+$ denote the two unbounded components
   of $\H\cap\partial T_k$. We claim that both $\alpha^-$ and
   $\alpha^+$ must intersect the vertical line segment
   $L := z_k + i[-2\pi,2\pi]$. Indeed, otherwise some $2\pi i\Z$-translate of
   $\alpha^-$ or $\alpha^+$ would separate $z_k$ from $\alpha^{\pm}$ in $\H$,
   which is not possible since $z_k$ belongs to the unbounded component of
   $T_k\cap\H$. 
   Hence it follows that the unbounded component of 
   $T_k\setminus L$, which contains $X_k(\cl{T}_k)$, is contained in
   $\H$. 

 In particular, we can
   pull back the set $X_k({T_k})$ 
   into $T_{k-1}$ using $\FInv{T_{k-1}}$. By expansivity of $F$,
 $\FInv{T_{k-1}}(X_k({T_k}))$   has distance at most
   $\pi$ from $z_{k-1}$. Continuing inductively,
   we obtain the sets
\[
   A_k := 
   X_0(\FInv{T_0}(X_1(\FInv{T_1}(\dots (X_{k-1}(\FInv{T_{k-1}}( X_k(\cl{T}_k))))\dots))))
\]
for $k\geq 1$; let $A_0=X_0(\cl{T}_0)$. 
  Each $\wh{A}_k\subset\Ch$ is a continuum, has distance at most $2\pi$ from
   $z_0$ and contains $\wh{A}_{k+1}$. 
  (Recall that $\wh{A}_k$ denotes the closure of $A_k$ in $\Ch$.)

Hence, the set $A'=\bigcap_{k\geq 0} \wh{A}_k$ has the same properties and there
  exists a component $A$ of $A' \setminus\{\infty\}$ with 
  $\dist(A,z_0)\leq 2\pi$. By definition, $A$ is closed and connected, and it 
  is unbounded by the Boundary Bumping Theorem 
  (Theorem \ref{thm:boundarybumping} in the appendix).
\end{proof}

\section{Functions Satisfying a Head-Start Condition}
\label{sec:head}

Throughout most of this section, 
 we will fix some function $F\in\Blog$.
 Fix an external address $\s$, and suppose that the 
 set $J_{\s}$ is a ray tail. Then $J_{\s}\cup\{\infty\}$ is homeomorphic to
 $[0,\infty]$, and as such possesses a natural total ordering.
 In this section, we will use a converse idea: we introduce a
 ``head-start condition'', which implies that the points in $J_{\s}$
 are totally ordered by their speed of escape, and deduce from this that
 $J_{\s}$ is (essentially) a ray tail. In the next section, we 
 develop sufficient conditions on $F$
 under which a head-start condition is satisfied.

\begin{defn}[Head-Start Condition]
\label{Def_HeadStart}
Let $T$ and $T'$ be tracts of $F$ and let
  $\phi:\R\to\R$ be a (not necessarily strictly) monotonically increasing 
  continuous function with
  $\phi(x)>x$ for all $x\in\R$. 
 We say that the pair $(T,T')$ satisfies
 the {\em head-start condition for $\phi$} if, for all
  $z,w\in \cl{T}$ with $F(z),F(w)\in \cl{T'}$, 
\[
\re w >\phi(\re z)\ \Longrightarrow\ \re F(w)>\phi(\re F(z))\;.
\]

An external address $\s$ satisfies the {\em head-start condition for
  $\phi$} if all consecutive pairs of tracts $(T_k,T_{k+1})$ satisfy the
  head-start condition for $\phi$, and if for all distinct
  $z,w\in J_{\s}$,  there 
  is $M\in\N$ such that 
  $\re F^{\circ M}(z)> \phi(\re F^{\circ M}(w))$ or
  $\re F^{\circ M}(w) > \phi(\re F^{\circ M}(z))$.

We say that $F$ satisfies a {\em head-start condition} if 
 every external address of $F$ satisfies the head-start condition for some $\phi$. 
 If the same function $\phi$ can be chosen for all external addresses, 
 we say that $F$ satisfies the {\em uniform head-start condition
 for $\phi$}.
\end{defn}

\begin{thm}[Ray Tails]
\label{thm:raytails}
Suppose that $F\in\Blog$ satisfies a head-start condition. Then for every 
escaping point $z$, there exists $k\in\N$ such that $F^{\circ k}(z)$ is on a ray tail $\gamma$. This ray tail is the unique arc in $J(F)$ connecting
  $F^{\circ k}(z)$ to $\infty$ (up to reparametrization).
\end{thm}

We devote the remainder of this section to the proof of Theorem \ref{thm:raytails}.

If $\s$ satisfies any head-start condition, the points in 
$J_{\s}$ are eventually ordered by their real parts: for any two points 
$z,w\in J_{\s}$, $F^{\circ k}(z)$ is to the right of $F^{\circ k}(w)$ for all sufficiently large $k$, or vice versa. 

\begin{deflem}[Speed Ordering]
\label{Def_SpeedOrder} \label{lem:totalorder}
Let $\s$ be an external address satisfying the head-start condition for
 $\phi$. For $z,w\in J_{\s}$, we say that $z\succ w$ if there exists $K\in\N$ 
 such that $\re F^{\circ K}(z)>\phi(\re F^{\circ K}(w))$. We extend this
 order to the closure $\widehat{J_{\s}}$ in $\Ch$ by the convention that
 $\infty\succ z$ for all $z\in J_{\s}$.
 
With this definition, $(\wh{J_{\s}},\succ)$ is a totally ordered space. Moreover, the order does not depend on $\phi$.
\end{deflem}
Note that if $z\succ w$, then $\re F^{\circ k}(z)>\phi(\re F^{\circ k}(w))$ for all $k\geq K$.
\begin{proof} By definition, $\re F^{\circ k}(z)<\phi(\re F^{\circ k}(z))$ for all
   $k\in\N$ and $z\in J_{\s}$. Hence ``$\succ$'' is non-reflexive. 

Let $a,b,c\in J_{\s}$ such that $a\succ b$ and $b\succ c$. Then, there exist
  $k,l\in\N$ such that $\re F^{\circ k}(a)>\phi(\re F^{\circ k}(b))$ and
  $\re F^{\circ l}(b)>\phi(\re F^{\circ l}(c))$. Setting $m:=\max\{k,l\}$, we 
  obtain from the head-start condition 
  that
  $\re F^{\circ m}(a)>\phi(\re F^{\circ m}(b))>
   \re F^{\circ m}(b)>\phi(\re F^{\circ m}(c))$. 
  Hence $a\succ c$ and ``$\succ$'' is transitive.

 By assumption, for any distinct $z,w\in J_{\s}$ there exists $k\in\N$ such that 
  $\re F^{\circ k}(w)> \phi(\re F^{\circ k}(z))$ or 
  $\re F^{\circ k}(z)> \phi(\re F^{\circ k}(w))$. It
   follows that any two distinct points are comparable under ``$\succ$''. 

 Furthermore, note that $w\succ z$ if and only if 
   $\re F^n(w)>\re F^n(z)$ for all sufficiently large $n$. This formulation
   is independent of $\phi$, proving the final claim.
\end{proof}

\begin{cor}[Growth of Real Parts]
\label{cor:growthofrealparts}
Let $\s$ be an external address that satisfies the head-start condition for
 $\phi$ and let $z,w\in J_{\s}$. If $w\succ z$, then
$w\in I(F)$. In particular, $J_{\s}\setminus I_{\s}$ consists of at most
 one point. 
\end{cor}
\begin{proof} This is an immediate corollary of Lemmas
  \ref{lem:realseparation} and \ref{lem:totalorder}. 
 \end{proof}

\begin{prop}[Arcs in $J_{\s}$]
\label{prop:CsArc}
Let $\s$ be an external address satisfying the head-start condition for
 $\phi$. Then the topology of $\wh{J_{\s}}$ as a subset of
 the Riemann sphere $\Ch$ agrees with the order topology induced by
 $\succ$. In particular,
 \begin{enumerate}
  \item every component of $\wh{J_{\s}}$ is homeomorphic to
         a (possibly degenerate) compact interval, and
  \item if $J^K_{\s}\neq\emptyset$ for $K$ as in Theorem
   \ref{thm:ripponstallard}, then $J_{\s}$ has a unique unbounded component, which is a closed arc to infinity.
 \end{enumerate}
\end{prop}
\begin{proof} Let us first show that $\id:\wh{J_{\s}}\to (\wh{J_{\s}},\succ)$ is continuous. Since $\wh{J_{\s}}$ is compact and the order topology on $\wh{J_{\s}}$ is Hausdorff, this will 
 imply that $\id$ is a homeomorphism and that both topologies agree. It suffices to show that sub-basis elements for the order topology of the form
  $U^-_a:=\{w\in J_{\s}\,: \, a\succ w\}$ and
  $U^+_a:=\{w\in \wh{J_{\s}}\,: \, w\succ a\}$ are open in $\wh{J_{\s}}$ for any $a\in \wh{J_{\s}}$.
  We will only give a proof for the sets $U^-_a$; the proof for $U^+_a$ is analogous.

Let $w\in U^-_a$ and choose $k\in\N$ minimal such that 
 $\re F^{\circ k}(a)> \phi(\re F^{\circ k}(w))$. Since $\phi, \re$ and $F^{\circ k}$ are continuous, this is true for a neighborhood $V$ of $w$. It follows that $V\cap \wh{J_{\s}}\subset U^-_a$, hence $U^-_a$ is a neighborhood of $w$ in $\wh{J_{\s}}$.
 
Thus the topology of $\wh{J_{\s}}$ agrees with the order topology.
 Every connected component $C$ of $\wh{J_{\s}}$ is compact;
 it follows from a well-known characterization of
 the arc (Theorem \ref{thm:arccharacterization} in the appendix)
that
 $C$ is either a point or an arc. This proves (a). To prove (b), observe that existence follows from Theorem \ref{thm:ripponstallard}, while uniqueness follows because $\infty$ is the largest element of $(\wh{J_{\s}},\succ)$.
\end{proof}

\begin{prop}[Points in the Unbounded Component of $J_{\s}$]
\label{prop:EscapingInCs}
Let $\s$ be an external address that satisfies the head-start condition 
 for $\phi$. Then there exists $K'\geq 0$ such that
 $J^{K'}_{\s}$ is either empty or 
 contained in the unbounded component of $J_{\s}$
 (and this component is a closed arc).
 The value $K'$ depends on $F$ and $\phi$, but not on $\s$.
\end{prop}
\begin{proof} We may assume without loss of generality that $F$ is normalized,
  i.e.\ $F\in\Blognorm$.
 Let $K$ be the constant from Theorem \ref{thm:ripponstallard}, set
 $K' := \max\{\phi(0)+1,K\}$ and let $z_0\in J^{K'}_{\s}$. 
For each $k\geq 0$, we let $z_k := F^{\circ k}(z_0)$ and consider the set
\[
	S_k:= \{w\in J_{\sigma^k(\s)}: w\succeq z_k\};.
\]
  By Proposition \ref{prop:CsArc}, each $S_k$ has a unique
   unbounded component $A_k$ that is a closed arc. By Theorem \ref{thm:ripponstallard}, $A_k$
satisfies $\dist(z_k,A_k)\leq 2\pi$. 

Let us show $A_k\subset \H$ for $k\geq 1$, 
 so that we may apply $F^{-1}$ to it. Indeed, if $w\in J_{\sigma^k(\s)}$ with $\re w \leq 0$, then the choice of $K'$ and monotonicity of $\phi$ yield that $\re z_k > \phi(0) \geq \phi(\re w)$, and therefore $z_k \succ w$. Thus, $w\not\in S_k$. We conclude that $F^{-1}_{T_{k-1}}(A_{k})\subset A_{k-1}$, because it is unbounded and contained in $S_{k-1}$. Since $F$ is expanding, this means that 
\[ 
	\dist(A_0,z_0)\leq 2^{-k}\dist(z_k,A_k)\leq 2^{-(k-1)}\pi 
\]
for all $k\geq 0$. Thus $z_0\in A_0$, as required. That $A_0$ is an arc follows from Proposition \ref{prop:CsArc}.
\end{proof}

\begin{proof}[Proof of Theorem \ref{thm:raytails}.]
Let $z$ be an escaping point for $F$ and $\s$ its external address. 
 By hypothesis, there exists $\phi:\R\to\R$ such that $\s$ satisfies the 
 head-start condition for $\phi$. If $K'$ is the constant from Proposition 
 \ref{prop:EscapingInCs}, then there exists $k\geq 0$ such that
 $F^{\circ k}(z)\in J^{K'}_{\s}$ and 
 $\gamma_k:=\{w\in I_{\sigma^k(\s)}\,:\, w \succeq F^{\circ k}(z)\}$ 
 is an injective curve connecting $F^{\circ k}(z)$ to $\infty$.
 Furthermore, since the order topology agrees with the usual topology on
 $J_{\s}\cup\infty$, $\gamma_k$ is unique with this property.

 To show that $\gamma_k$ is a ray tail, we still need to show that 
  \[ \lim_{m\to\infty} \inf_{w\in\gamma_k} \re F^{\circ m}(w) = \infty. \]
  This follows from the head-start condition. 
  Indeed, for $w\in\gamma_k$ and $m\in\N$, we have 
  $\re F^{\circ m}(w)\geq \inf \{\phi^{-1}(\re F^{\circ (k+m)}(z))\}$, 
  because $w\succ z$ or $w=z$ (we have to take the infimum because $\phi$ 
  need not be invertible). This lower bound tends to infinity as 
  $m\to \infty$. 
\end{proof}

\begin{thm}[Existence of Absorbing Brush]
\label{thm:absorbing}
 Suppose that $F\in\Blog$ satisfies a head-start condition. Then 
  there exists a closed $2\pi i$-periodic 
  subset $X\subset J(F)$ with the following
  properties:
  \begin{enumerate} 
   \item $F(X)\subset X$;
   \item each connected component $C$ of $X$ is a closed arc to infinity all of
     whose points except
     possibly the finite endpoint escape; \label{item:arcs}
   \item every escaping point of $F$ enters $X$ after
      finitely many iterations. If $F$ satisfies the uniform head-start condition for some function, then there exists $K'>0$ such that $J^{K'}(F)\subset X$. \label{item:absorbing}
  \end{enumerate}

  If, additionally, $F$ is of disjoint type, then we may choose $X=J(F)$. 
\end{thm}
\begin{remark}
 It is not difficult to show that the set $X$ is in fact a 
  \emph{Cantor Bouquet}; i.e.\ homeomorphic to a ``straight brush'' in the
  sense of Aarts and Oversteegen \cite{aartsoversteegen}. However, we will not
  give a proof here.
\end{remark}
\begin{proof} Let $X$ denote the union of all the unbounded components of
  $J(F)$. By the Boundary Bumping Theorem \ref{thm:boundarybumping},
  $\wh{X}$ is the connected component of the compact set
  $J(F)\cup\{\infty\}$ that contains $\infty$, 
  hence $X\subset\C$ is a closed set. Clearly $X$ is
  $F$-invariant, and satisfies (\ref{item:arcs}) and
  (\ref{item:absorbing}) by Propositions
  \ref{prop:CsArc} and \ref{prop:EscapingInCs} (recall that the choice of $K'$ did not depend on the external address). 

 Recall that $F:\T\to H$ is of disjoint type if $\cl{\T}\subset H$. 
  In this case, 
  $J(F)\cup\{\infty\}$ is connected, since it is the 
  nested intersection of the unbounded compact connected sets
  $F^{-n}(\cl{H})\cup\{\infty\}$. 
  Hence it follows from the above that $X=J(F)$. 
\end{proof}

\section{Geometry, Growth \& Head-Start}
\label{sec:growth}

This section discusses geometric properties of tracts that imply a head-start 
 condition. Moreover, we show that (compositions of) functions of finite order
 satisfy these properties, hence completing the proof of Theorem
 \ref{thm:positive}.

Let $K>1$ and $M>0$. We say that $\s$ satisfies the 
  \emph{linear head-start condition} with constants
   $K$ and $M$ if it satisfies the head-start condition
for
\[
	\phi(t) := K\cdot t^+ + M\;,
\]
where $t^+=\max\{t,0\}$. 

We will restrict our attention to functions whose tracts do not grow too quickly in the imaginary direction.

\begin{defn}[Bounded Slope]
\label{defn:boundedslope}
Let $F\in\Blog$. We say that the tracts of $F$ have bounded slope (with
constants $\alpha,\beta>0$) if 
\[
	|\im z - \im w| \leq
          \alpha\,\max\{\re z, \re w,0\} + \beta
\]
whenever $z$ and $w$ belong to a common tract of $F$. We denote the class of 
 all functions with this property by $\Blog(\alpha,\beta)$, and use
 $\Blognorm(\alpha,\beta)$ to denote those that are also normalized.
\end{defn}
\begin{remark}
By property (\ref{item:expinjective}) in the definition of $\Blog$, this 
 condition is equivalent to the existence of
a curve $\gamma:[0,\infty)\to\T$ with $|F(\gamma(t))|\to\infty$ and
$\limsup |\im\gamma(t)|/\re\gamma(t) < \infty$. Hence if one tract of $F$ has bounded slope, then all tracts do.
\end{remark}

The bounded slope condition means that the absolute value of a point is 
 proportional to
 its real part. As we see in the next lemma, this easily implies that
 the second requirement of a linear head-start condition, that any two
 orbits eventually separate far enough for one to have a head-start over
 the other, is automatically satisfied when the tracts have bounded slope.

\begin{lem}[Linear Separation of Orbits]
\label{lem:linearseparation}
Let $F\in\Blognorm$, and let $\alpha,\beta>0$. Let $T$ be a tract of $F$,
 and suppose that $z,w\in \cl{T}$ satisfy $\re F(w)\geq \re F(z)$ and
 $|\im F(w) - \im F(z)| \leq \alpha \re F(w) + \beta$.
 \begin{enumerate}
  \item \label{item:realsepgeneral}
    There exists a constant
     $\delta=\delta(\alpha,\beta)$, depending only on $\alpha$ and $\beta$,
      with the following property: if $|z-w|\geq \delta$, then 
     \[ \re F(w) > e^{|z-w|/16\pi} \re F(z). \]
  \item \label{item:realseplinear}
    Let $K\geq 1$ and $Q\geq 0$. Then there is a constant 
    $\delta=\delta(\alpha,\beta,K,Q)$ with the following property:
   if $|z-w| \geq \delta$, then 
     \[ \re F(w) > K\re F(z) + |z-w| + Q. \]
 \end{enumerate}

 In particular, suppose that $F\in\Blognorm(\alpha,\beta)$, and let $\s$
  be an external address. If $z,w\in J_{\s}$ with $|z-w|\geq 
   \delta(\alpha,\beta,K,0)$, then 
\[ 
	\re F^{\circ k}(z) > K \re F^{\circ k}(w) + |z-w| 
                                                       \quad\text{or}\quad
      	\re F^{\circ k}(w) > K \re F^{\circ k}(z) + |z-w| 
\]
 for all $k\geq 1$.
\end{lem} 
\begin{proof} 
  Set $\delta' := \alpha + \beta + 2$ and
   $\delta := \max\{\delta', 16\pi\log\delta'\}$. 
   The hypotheses on $z$ and $w$ imply that 
   $|F(w)-F(z)| \leq (\alpha+1)\re(F(w)) + \beta$. By
   expansivity of $F$, we have $|F(w)-F(z)|\geq 2\delta' > \alpha + 1 + \beta$,
   and thus
   $\re F(w) >  1$.
   We hence conclude that $|F(w) - F(z)|\leq \delta' \re F(w)$. 
   Because $|z-w|\geq 2$, Lemma \ref{lem:expansion} now yields 
  \begin{equation}
    \label{eqn:realsep}
	  \re F(w) \geq \frac{|F(w)-F(z)|}{\delta'} \geq
               \frac{\exp(|w-z|/8\pi)}{\delta'}\cdot \re F(z) > 
               e^{\frac{|w-z|}{16\pi}}\cdot\re F(z)\;,
   \end{equation}
 because $\exp(x/8\pi)/\delta'>\exp(x/16\pi)$ for all $x>16\pi\log\delta'$. 
 This proves part (\ref{item:realsepgeneral}). 

To prove part (\ref{item:realseplinear}), we now choose
  $\delta\geq \delta(\alpha,\beta)+1/2$ 
  sufficiently large that all $x\geq \delta-1/2$ satisfy 
   $e^{x/16\pi} > x + K + Q + 1/2$.
 
Let $z'\in T$ be the point with 
  $\re F(z')=\max(1,\re F(z))$ and $\im F(z')= \im F(z)$. Then
  $|z-z'|\leq 1/2$ by expansivity of $F$, so 
  we can apply~\eqref{eqn:realsep} to $z'$ and $w$:
   \begin{align*}
      \re F(w) &> e^{\frac{|w-z'|}{16\pi}}\cdot \re F(z') >
          (|w-z'| + K + Q + 1/2)\cdot \re F(z') \\
         &\geq
             K\re F(z') + |w-z'| + Q + 1/2 \geq
             K\re F(z) + |w-z| + Q. 
   \end{align*}

 The final claim follows from (\ref{item:realseplinear}) 
  by induction. 
\end{proof} 
\begin{remark}[Remark 1]
 The lemma shows that, if $F\in\Blognorm(\alpha,\beta)$ 
  satisfies the linear head-start condition for some $K$ and $M$, then $F$
  satisfies the linear head-start condition for \emph{all} 
  $\wt{K}\geq K$ and $M\geq \max(M,\delta(\alpha,\beta))$. 
\end{remark}
\begin{remark}[Remark 2] 
  By the final claim of the Lemma, 
  if $F\in\Blognorm(\alpha,\beta)$, then we only need to
  verify the first requirement of a linear head-start condition:
  if $w$ is ahead of $z$ in terms of real parts,
  the same should be true for $F(w)$ and $F(z)$. Note that this
  condition is not dynamical in nature;
  rather, it concerns the mapping behavior of the conformal map
  $F:T\to\H$. As such, it is not too difficult to translate it 
  into a geometrical condition. Roughly speaking,
  the tract should not ``wiggle'' in the sense of
  first growing in real parts to reach the larger point $w$, then
  turning around to return to $z$, until finally starting to grow again.
 (We exploit this idea in Section~\ref{sec:counter} to construct 
  a counterexample; compare also Figure \ref{Fig:CounterWiggle}). The precise 
  geometric condition 
  is as follows.
\end{remark}

\begin{defn}[Bounded Wiggling]
\label{Def_BoundedWiggle}
Let $F\in\Blog$, and let $T$ be a tract of $F$. We say that
$T$ has \emph{bounded wiggling} if there exist $K>1$ and $\mu>0$
such that for every $z_0\in \cl{T}$, every point $z$ on the hyperbolic geodesic of $T$ that connects $z_0$ to $\infty$ satisfies
\[
	 (\re z)^+ > \textstyle{\frac{1}{K}}\re z_0 - \mu\;.
\]
We say that $F\in\Blog$ has {\em uniformly bounded wiggling} if the wiggling of all tracts of $F$ is bounded by the same constants $K,\mu$.
\end{defn}

\begin{prop}[Head-Start and Wiggling for Bounded Slope]
\label{prop:geometryheadstart}
Let $F\in\Blognorm(\alpha,\beta)$, and let
$K>1$. Then the following are equivalent: 
\begin{enumerate}
   \item For some $M>0$, $F$ satisfies the uniform linear head-start
     condition with constants $K$ and $M$. 
      \label{item:ssatisfieslinearheadstart}
   \item For some $\mu>0$, the tracts of $F$ have uniformly bounded
     wiggling with constants $K$ and $\mu$.
      \label{item:tractshaveboundedwiggling}
   \item For some $M'>0$, the following holds. If $T$ is a tract of $F$ and
     $z,w\in \cl{T}$ with 
     $\re w > K(\re z)^+ + M$ and $|\im F(z) - \im F(w)|\leq
     \alpha\max\{\re F(z),\re F(w)\} +\beta$, then 
     $\re F(w) > K\re F(z) + M$.  \label{item:strongerlinearheadstart}
\end{enumerate}
\end{prop}
\begin{proof} Condition (\ref{item:strongerlinearheadstart}) implies
   (\ref{item:ssatisfieslinearheadstart}) by definition.
 To show that (\ref{item:tractshaveboundedwiggling}) implies
  (\ref{item:strongerlinearheadstart}), let us set
  $\wt{M} := K\cdot(\mu + 2\pi(\alpha+\beta))$ and 
  define $M:=\max(\delta,\wt{M},1)$, where
  $\delta=\delta(\alpha,\beta, K, 0)$ 
  is  the constant from Lemma 
 \ref{lem:linearseparation}. Let $T$ be a tract of $F$ and let
  $z,w\in \cl{T}$ be as in 
  (\ref{item:strongerlinearheadstart}). Then $|z-w| > M$ and, 
  by Lemma \ref{lem:linearseparation}
  (\ref{item:realseplinear}), it suffices to show that 
  $\re F(w) \geq \re F(z)$.

 So suppose by way of contradiction that $\re F(z) > \re F(w)$. Since
  $M\geq 1$, we see from Lemma \ref{lem:linearseparation} that  
  $\re F(z)\geq |z-w| > M \geq 1$. Set 
  $\Gamma := \{F(w)+t\,:\,t\geq 0\}$ and 
  $\gamma:=F_{T}^{-1}(\Gamma)$; in other words, 
  $\gamma$ is the geodesic
  of $T$ connecting $w$ to $\infty$. The assumption on $F(z)$ and $F(w)$
 ensures that
\[
	\dist_{T}(z,\gamma) = 
        \dist_{\H} (F(z),\Gamma) \leq 
         \frac{|\im F(z)-\im F(w)|}{\re F(z)}\leq \alpha + \beta\;.
\]
  Therefore, $\dist(z,\gamma)\leq 2\pi(\alpha+\beta)$ 
  by the standard estimate (\ref{eqn:standardestimate}), and consequently 
  $\re z + 2\pi(\alpha+\beta) \geq \min_{\zeta\in\gamma}\re \zeta$. By
  the bounded wiggling condition, we also have
  $(\re\zeta)^+ \geq \frac{1}{K}\re w - \mu$ for all $\zeta\in\gamma$. 
  Thus
\[ \re w \leq K((\re z)^+ + \mu + 2\pi(\alpha+\beta)) < 
              K(\re z)^+ + \wt{M} \leq K\re z + M \;,
\]
a contradiction. 

Now suppose that
(\ref{item:ssatisfieslinearheadstart}) holds. Let $T$ be a tract and 
 $z\in \cl{T}$. 
 We use the results of the previous section. These imply, in particular, that 
 there is an injective curve $\Gamma\subset I(F)\cap\H$ such that
 $\Gamma\cup \{\infty\}$ is an arc. Since
 $I(F)$ is $2\pi i$-periodic, we may choose $\Gamma$ such that
 $\dist(F(z),\Gamma)<\kappa$, where $\kappa>0$ is a constant that is independent
 of $T$ and $z$. Pulling back, we obtain a point $\zeta\in T$ that can be
 connected to $z$ by a curve of Euclidean length at most $\kappa/2$,
 and to $\infty$ by an injective curve $\gamma\subset I(F)$. Recall
 that $\zeta\prec w$ for all $w\in\gamma$ (where $\prec$ is the speed ordering
 from the previous section). By definition of $\prec$, we have 
\[
	(\re w)^+ \geq \frac{\re \zeta}{K}-\frac{M}{K}
\]
 for all $w\in \gamma$. 
Hence there exists a curve $\gamma'\subset T$ connecting 
  $z$ to $\infty$ such that for every $w\in\gamma'$,
\[
	(\re w)^+ \geq \frac{\re \zeta}{K}-\frac{M}{K}-\kappa/2 
    \geq \frac{\re z}{K}-\frac{\kappa}{K}-\frac{M}{K}-\kappa/2\;.
\]
Now (\ref{item:tractshaveboundedwiggling})
   follows easily by general principles of hyperbolic
 geometry (see Lemma \ref{lem:boundedwiggling} in the appendix).
\end{proof}

We now  consider functions of finite order. 
\begin{defn}[Finite Order]
\label{Def_FiniteOrder}
We say that $F\in\Blog$ has \emph{finite order} if
\[ 
	\log \re F(w) = O(\re w) 
\]
as $\re w\to\infty$ in $\T$.
\end{defn}
Note that this definition ensures that $f\in\B$ has finite order (i.e.\ 
\[
\lim_{r\to\infty}\sup_{|z|=r}\frac{\log\log |f(z)|}{\log|z|}<\infty)
\]
if and only if any logarithmic transform $F\in\Blog$ of $f$
 has finite order in the sense of Definition \ref{Def_FiniteOrder}.

\begin{thm}[Finite Order Functions have Good Geometry]
   \label{thm:finiteorder}
 Suppose that $F\in\Blognorm$ has finite order. Then the tracts of
  $F$ have bounded slope and (uniformly) bounded wiggling.
\end{thm}
\begin{proof} 
By the Ahlfors non-spiralling theorem (Theorem \ref{thm:spiral}), 
 $F\in\Blognorm(\alpha,\beta)$ for some constants $\alpha,\beta$.
 By the finite-order condition, there are $\rho$ and $M$ such that
$\log \re F(z) \leq \rho \re z + M$ for all $z\in\T$. Let $T$ be a tract of $F$ and $z\in \cl{T}$. 

Suppose first that $\re F(z) \geq 1$. Consider the geodesic $\gamma(t) := F^{-1}_T( F(z) + t )$ (for $t\geq 0$). Since the hyperbolic distance between $z$ and $\gamma(t)$ is at most $\log(1+t)$, we have 
\[
	 \re z - \re\gamma(t) \leq
     2\pi \log(1+t) \leq 2\pi \log \re F(\gamma(t)) \leq
                2\pi(\rho\re\gamma(t) + M) \;. 
\]
In other words, $\re z \leq (1+2\pi\rho)\re\gamma(t) + 2\pi M$, i.e.\
\[
    \re \gamma(t) \geq \frac{1}{1+2\pi\rho} \re z - \frac{2\pi M}{1+2\pi\rho} \;. 
\]
Since $z$ was chosen arbitrarily, $F$ has uniformly bounded wiggling with constants $1/(1+2\pi\rho)$ and $2\pi M/(1+2\pi\rho)$.

If $\re F(z)<1$, then by expansivity of $F$ 
 we can connect $z$ to a point $w\in T$ with $\re F(w)\geq 1$ 
 by a curve of bounded Euclidean diameter. 
\end{proof}

To complete the proof of Theorem \ref{thm:positive}, it only 
 remains to show that
 linear head-start conditions are preserved under composition. In logarithmic
 coordinates, this is given by the following statement;
 let $\tau_a(z)=z-a$ for $a\ge 0$ and $\H_a:=\{z\in\C\colon\re(z)>a\}$.

\begin{lem}[Linear Head-Start is Preserved by Composition]
\label{lem:composition}
Let $F_i\colon \T_{F_i}\to\H$ be in $\Blognorm$, for 
 $i=1,2,\dots,n$. Then there is an $a\ge 0$ so that 
 $G_a:=\tau_a\circ F_n\circ\dots\circ F_1\in\Blognorm$ on appropriate tracts 
 $\T_a\subset \T_{F_1}$, so that $G_a$ is a conformal isomorphism from each 
 component of $\T_a$ onto $\H$. 
 If $F_1$ has bounded slope and all $F_i$ satisfy 
 uniform linear head-start conditions,
  then $G_a$ also has bounded slope and satisfies a 
  uniform linear head-start condition.
\end{lem}
\begin{proof} 
There is an $a_2\ge 0$ so that $F_2^{-1}(\H_{a_2})\subset\H$; there is an $a_3\ge 0$ so that $F_3^{-1}(\H_{a_3})\subset \H_{a_2}$, etc.. Finally, there is an $a=a_n\ge 0$ so that $(F_n\circ\dots\circ F_1)^{-1}$ is defined on all of $\H_a$. Let $\T_a:=(F_n\circ\dots\circ F_1)^{-1}(\H_a)\subset\T_{F_1}$. Then $F_n\circ\dots\circ F_1$ is a conformal isomorphism from each component of $\T_a$ onto $\H_a$, and the first claim follows. In particular, the tracts of $G_a$ have bounded slope.

For $i=1,\dots,n$, let $K_i$ and $M_i$ be the constants for the linear 
 head-start condition of $F_i$, and set $K:=\max_i\{K_i\}$ 
 and $M:=\max(\delta,\max_i M_i)$, where $\delta=\delta(\alpha,\beta,K,0)$ 
 is the constant from Lemma \ref{lem:linearseparation}.
Let $T$ be a tract of $F_i$ and $w,z\in T$, such that 
 $\re w > K \re z + M$. Then, $|w-z|\geq \re w - \re z > M = \delta$, and 
 Lemma \ref{lem:linearseparation} gives that
\[
	\re F_i(w) > K \re F_i(z) + M \quad\text{or}\quad \re F_i(z) > K \re F_i(w) + M\;.
\]
Since $F_i$ satisfies a head-start condition, the first inequality must hold. 
 Hence, all $F_i$ satisfy a linear head-start condition with constants $K,M$, 
 and it is now easy to see that $G_a$ does, too.
\end{proof}

\begin{proof}[%
    Proof of Theorem \ref{thm:positive} and Corollary \ref{cor:mero}]  
  Let $f_1,\dots, f_n\in\B$ be functions of finite order. By applying 
  a suitable affine change of variable, to all $f_i$, we may assume without
  loss of generality that each $f_i$ has a \emph{normalized} logarithmic
  transform $F_j\in\Blog$. 
  By Theorem
  \ref{thm:finiteorder}, each $F_j$ satisfies a linear head-start condition, and by 
  Lemma~\ref{lem:composition},  
  $G_a:=\tau_a\circ F_n\circ\dots\circ F_1\in\Blog$ satisfies a linear 
  head-start condition. (The purpose of $\tau_a$ is only to arrange the maps 
  so that their image is all of $\H$.) Now, on a sufficiently restricted 
  domain, $F:=G_a\circ \tau_a^{-1}$ is a logarithmic transform of 
  $f=f_n\circ\dots\circ f_1$ and satisfies a linear head-start condition. 
  Thus every escaping point of $F$, and hence of $f$, is eventually mapped
  into some ray tail. By Proposition \ref{prop:classification}, this
  completes the proof of Theorem \ref{thm:positive}.

 The proof of Corollary \ref{cor:mero} is analogous.
  (Recall that the order of a meromorphic function is defined in
   terms of its Nevanlinna characteristic. However,
   if $f$ has finite order, then it is well-known that
   the restriction of $f$ to its logarithmic
   tracts will also have finite order in the previously defined 
   sense.)
\end{proof}

\begin{remark}
 If our goal was only to prove Theorem \ref{thm:positive} and Corollary
  \ref{cor:mero}, a somewhat faster route would be possible
  (compare \cite[Chapter 3]{guenterthesis}). For example, the linear 
  head-start condition can be verified directly for functions
  of finite order, without explicitly considering the geometry of
  their tracts. Also, the bounded slope condition can be used to simplify
  the proof of Theorem \ref{thm:raytails} in this context, eliminating e.g.\
  the need for Theorem \ref{thm:ripponstallard}. We have chosen the
  current approach because it provides both additional information and
  a clear conceptual picture of the proof. 
\end{remark}

 Let us collect together some of the results obtained in this and the
  previous section for future reference.
 \begin{cor}[Linear Head-Start Conditions] 
  Let $\Hlog$ consist of all functions $F\in\Blog$ that satisfy a 
   uniform linear head-start condition and have tracts of bounded slope. 
  \begin{enumerate}
   \item The class $\Hlog$ contains all function $F\in\Blognorm$ of 
     finite order.
   \item The class $\Hlog$ is closed under composition.
   \item If $F\in\Hlog$, then there is some $K>0$ such that every point of
     $J^K(F)$ can be connected to infinity by a curve in $I(F)$. 
   \item If $F\in\Hlog$ is of disjoint type, then every component of $J(F)$
     consists of a dynamic ray together with a unique landing point. 
  \end{enumerate}
 \end{cor}
\begin{remark}[Remark 1]
 Here closure under composition should be understood in the sense of 
  Lemma \ref{lem:composition}. I.e., given functions $F_1,\dots,F_n\in\Hlog$, 
    the
  function $F_1\circ\dots\circ F_n$ belongs to $\Blog$ 
  after a suitable restriction and conjugation with a translation;
  this map then also belongs to $\Hlog$.  
\end{remark}
\begin{remark}[Remark 2]
 It is easy to see that the class $\Hlog$ is also closed under
  \emph{quasiconformal equivalence near infinity} in the sense of
  \cite{boettcher}.
\end{remark}
\begin{proof}
 The first claim is a combination of Theorem \ref{thm:finiteorder} and 
  Proposition \ref{prop:geometryheadstart}. The second follows from 
  Lemma \ref{lem:composition}, the third from
  Proposition \ref{prop:EscapingInCs}, and the final claim
  from Theorem \ref{thm:absorbing}. 
\end{proof}

In order to apply our results to functions in $\Blog$ that are of disjoint type
 but not necessarily normalized, we need to be able to verify 
 linear head-start conditions for these functions. The following 
 lemma allows us to do this using the results we proved for 
 normalized functions.
\begin{prop}[Disjoint-type maps and linear head-start]
   \label{prop:disjointheadstart}
 Let $F:\T\to H$ be a disjoint-type map in $\Blog(\alpha,\beta)$, and let
  $R>0$ such that $\H_R\subset H$.
  Then $F$ satisfies a uniform
  linear head-start condition if and only if
  the map $\wt{F}:= F|_{F^{-1}(\H_R)}$ satisfies a uniform linear head-start
  condition.
\end{prop}
\begin{proof}
  The ``only if'' direction is trivial, so suppose that
   $\wt{F}$ satisfies a uniform linear head-start condition.
   We may assume that $R$ is sufficiently large that 
   (\ref{eqn:expansion}) holds whenever $\re F(z) \geq R$;
   set $\wt{V} := F^{-1}(\H_R)$. 
   Then the map $G := \wt{F}(z+R)-R$ is an element of
   $\Blognorm(\alpha,\beta+R\alpha)$ and satisfies a uniform linear head-start
   condition.

 Define $C$ to be the maximal hyperbolic diameter, in $H$, of a
  component of $\cl{\T}\setminus \H_R$. 
  (This is a finite number because
  $F$ is of disjoint type.) Applying Lemma \ref{lem:realseparation} and
  Proposition \ref{prop:geometryheadstart} (\ref{item:strongerlinearheadstart})
  to $G$, and translating the
  results back to $\wt{F}$, we see that there are constants $K$ and $M$
  with the following property. Suppose that $z,w$ belong to a component
  $\wt{T}$ of $\wt{\T}$ and
   $F(z),F(w)$ belong to a component $T'$ of $\T$. If
   $\re w > (\re z)^+ + M$, then
    \begin{equation} \label{eqn:headstartforFtilde}
     \re F(w) > K \re F(z) + M + R + 4\pi C. 
    \end{equation}

  We
   shall show that $F$ satisfies the uniform head-start condition for
   $\phi(t) = Kt^+ + M + 4\pi C$. 
  Let $T$ and $T'$ be tracts of $F$, and suppose that
  $z,w\in T$ with $F(z), F(w)\in T'$ and $\re w > \phi(\re z)$.

  By definition of $C$, we can find a point $z'\in T$ such that
   $\re F(z') = \max(\re F(z),R)$, $F(z)\in T'$ and 
   $\dist_H(F(z'),F(z))\leq C$.
   Because $F$ is a conformal isomorphism, we have
   $\dist_T(z',z)\leq C$, and by the standard estimate
   (\ref{eqn:standardestimate}), $|z'-z| \leq 2\pi C$. There is also a
   point $w'$ with the corresponding properties for $w$. We now apply
   (\ref{eqn:headstartforFtilde}) to $z'$ and $w'$ to see that
     \begin{align*} 
      \re F(w) \geq \re F(w') - R &> 
      K\cdot \re F(z') + M + 4\pi C \\ &\geq
      K\cdot \re F(z) + M + 4\pi C = \phi(\re F(z)). \end{align*}   

 This proves the first requirement of the head-start condition. The second
  follows easily from the fact that $F$ uniformly expands the hyperbolic
  metric (Lemma \ref{lem:hyperbolicexpansion}); 
  we leave the details to the reader. 
\end{proof}

We can use the above lemma to describe the Julia sets of certain
 hyperbolic functions that are compositions of finite-order functions
 in class $\B$. As mentioned in the introduction, this
 has been proved by Bara\'nski \cite{baranskihyperbolic} when
 $f$ is of finite order.

 \begin{thm}[Disjoint-type maps] \label{thm:baranskitype}
  Let $f=f_1\circ f_2 \circ \dots \circ f_n$, where $f_i\in\B$ for all $i$,
   and all $f_i$ have finite order. Suppose that $S(f)\subset F(f)$ and that
   $F(f)$ consists only of
   the immediate basin of an attracting fixed point of $f$.
 
 Then every component of $J(f)$ is a dynamic ray together with 
   a single landing point; in particular, every point of $I(f)$ is on 
   a ray tail of $f$. 
 \end{thm}
 \begin{proof}
  The assumptions imply that there is a bounded Jordan domain $D$ such that
   $S(f)\subset D$ and $f(\cl{D})\subset D$. (This is a simple exercise.)
   Using this domain in the
   definition of a logarithmic transform $F$ of $f$, we see that $F$ is
   of disjoint type, and that $\exp(J(F))= J(f)$. As in the
   proof of Theorem \ref{thm:positive}, a suitable restriction
   of $F$ satisfies a uniform linear head-start condition.
   The claim now follows
   from Proposition \ref{prop:disjointheadstart}
   and  Theorem \ref{thm:absorbing}.
 \end{proof}

 To conclude the section, let us comment on the issue of ``random iteration'',
  where we are considering
  a sequence $\F = (F_0,F_1,F_2,\dots)$ 
  of functions, and study the corresponding ``escaping set''
  set $I(\F)=\{z\in\C:\F_n(z)\to\infty\}$, where
  $\F_n=F_n\circ F_{n-1}\circ \dots \circ F_0$.
  (Now the pairs
  $(T,T')$ will consist of a tract $T$ 
  of $F_{k}$ and a tract $T'$ of $F_{k+1}$, etc.) Our proofs carry through
  analogously in this setting. In particular, if all tracts of all
  $F_j$ have uniformly bounded wiggling
  and uniformly bounded slope, then again for every $z\in I(\F)$, there is 
  some iterate $\F_n(z)$ that can be connected to infinity by a curve in
  the escaping set $I(F_n,F_{n+1},\dots)$.

\section{Counterexamples}
\label{sec:counter}

This section is devoted to the proof of Theorem \ref{thm:counterexample};
 that is, the construction of a counterexample to the strong form 
 of Eremenko's Conjecture. As mentioned in the previous section, such
 an example will be provided by a function with a tract that has
 sufficiently large ``wiggles''. 

We begin by formulating the exact properties our counterexample should have.
 Then we construct a tract (and hence a function $F\in\Blog$) with
 the required properties. Finally, we show how to realize such a tract
 as that of an entire function $f\in\B$, using a function-theoretic principle.

To facilitate discussion in this and the next section, let us call
 an unbounded Jordan domain $T$ a \emph{tract} if the real parts of $T$ are
 unbounded from above and the translates $T+2\pi i n$ (for $n\in\Z$) have
 pairwise disjoint closures in $\C$.

\begin{thm}[No Curve To Infinity]
\label{Thm:NoCurveToInfinity}
 Let $T\subset\H$ be a tract, and let
  $F_0\colon T\to\H$ be a Riemann map, with continuous extension
  $F_0\colon \wh{T}\to\wh{\H}$ given by Carath\'eodory's Theorem. Suppose
  that the following hold:
\begin{enumerate}
\item $F_0(\infty)=\infty$;  \label{Item:Infinity}
\item $|\im z - \im z'|<H$ for some $H<2\pi$ and all $z,z'\in T$;
\label{Item:Domain}%
\item
there are countably many disjoint hyperbolic geodesics $C_k,\sep C_k\subset T$, for $k=0,1,\dots$, so that all $F_0(C_k)$ and $F_0(\sep C_k)$ are semi-circles in $\H$ centered at $0$ with radii $\rho_{k+1}$ and $\sep \rho_{k+1}$ so that  $\rho_1<\sep \rho_1<\rho_2<\dots$;
\label{Item:Geodesics}
\item
all $\rho_k+H< \sep \rho_k/2$ and all $\sep \rho_k+H<\rho_{k+1}/2$;
\label{Item:RhoEstimates}
\item
all  $C_k$ and $\sep C_{k}$ have real parts strictly between $\sep \rho_k+H$ and $\rho_{k+1}/2$;
\label{Item:RealParts}
\item
all points in the unbounded component of $T\sm \sep C_k$ have real parts greater than $\sep \rho_k$;%
\label{Item:NotFarBack}%
\item
every curve in $T$ that connects $C_k$ to $\sep C_k$ intersects the line $\{z\in\C\colon\re z=\rho_k/2\}$.
\label{Item:TurnPoints}
\end{enumerate}
Define $T_n:=T+2\pi i n$ for $n\in\Z$ and $\T:=\bigcup_n T_n$, define Riemann maps $F_n\colon T_n\to\H$ via $F_n(z):=F_0(z-2\pi in)$, 
and define a map $F\colon\T\to\H$ that coincides on $T_n$ with $F_n$ for each $n$. 

Then the set 
  $J:=J(F)=\{z\in T\colon F^{\circ k}(z)\in \T \mbox{ for all $k$}\}$ contains no curve to $\infty$.
\end{thm}

\begin{figure}[hbt]
\begin{center}
\setlength{\unitlength}{1cm}
\begin{picture}(12,4)
\put(0,0){\includegraphics[viewport=000 000 800 400,width=12cm]{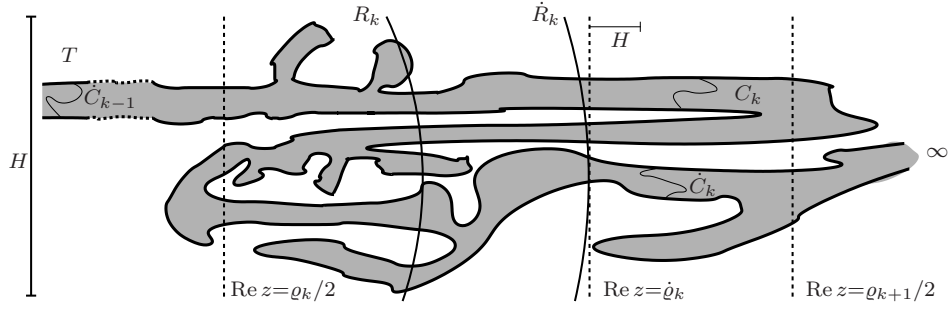}}
\put(0.8,2.65){$\scriptstyle\sep{C}_{k-1}$}
\put(9.45,2.7){$\scriptstyle C_k$}
\put(8.85,1.47){$\scriptstyle \sep{C}_{k}$}
\put(0.5,3.2){$\scriptstyle T$}
\put(-0.2,1.8){$\scriptstyle H$}
\put(12,1.95){$\scriptstyle\infty$}
\put(4.37,3.75){$\scriptstyle R_k$}
\put(6.75,3.75){$\scriptstyle \sep{R}_k$}
\put(2.75,0.1){$\scriptstyle \re z = \rho_k/2$}
\put(7.7,0.1){$\scriptstyle \re z = \sep{\rho}_k$}
\put(10.4,0.1){$\scriptstyle \re z = \rho_{k+1}/2$}
\put(7.8,3.4){$\scriptstyle H$}
\end{picture}
\end{center}
\caption{\label{Fig:CounterWiggle} A tract that satisfies the conditions of Theorem \ref{Thm:NoCurveToInfinity}. The figure is not to scale, as can be seen from the horizontal and vertical dimensions of the length $H$.}
\end{figure}

\begin{proof}
Since the $T_n$ have disjoint closures, $F$ extends continuously to $\cl{\T}$.
Let $R_k$ and $\sep R_k$ be semicircles in $\H$ centered at $0$ with radii $\rho_k$ and $\sep \rho_k$, respectively.

Every $z\in J$ has an external address $\s=T_{s_0}T_{s_1}T_{s_2}\dots$ so that 
 $F^{\circ k}(z)\in T_{s_k}$ for all $k$. Clearly, all points within any 
 connected component of $J$ have the same external address, so we may fix an 
 external address $\s$ and show that the set $J_{\s}$ (i.e., the points in $J$ with address $\s$) contains no  curve to $\infty$. We may assume that there is an orbit $(w_k)$ with
  external address $\s$ (if not, then we have nothing to show). 

For simplicity, we write $\Csep{m}{k}$ for $C_m+2\pi i s_k$ and $\Csepp{m}{k}$ for $\sep C_m+2\pi i s_k$.

\begin{claim}[Claim 1]
There is an $m\ge 0$ so that for all $k\ge 0$, $\Csepp{m+k}{k}$ separates $w_k$ from $\infty$ within $T_{s_k}$, and $|w_{k+1}| < \sep \rho_{m+k+1}$.
\end{claim}
\begin{subproof}
We prove this claim by induction, based on Condition~(\ref{Item:NotFarBack}): some $\Csepp{m}{0}$ separates $w_0$ from $\infty$. For the inductive step, suppose that $\Csepp{m+k}{k}$ separates $w_k$ from $\infty$ within $T_{s_k}$. Then $\sep R_{m+k+1}$ separates $w_{k+1}$ from $\infty$ within $\H$, i.e., $\re w_{k+1}\le |w_{k+1}|<\sep \rho_{m+k+1}$ (Condition~(\ref{Item:Geodesics})). By Condition~(\ref{Item:NotFarBack}), it follows that $w_{k+1}$ is in the bounded component of $T_{s_{k+1}}\sm \Csepp{m+k+1}{k+1}$, so $\Csepp{m+k+1}{k+1}$ separates $w_{k+1}$ from $\infty$ within $T_{s_{k+1}}$, and this keeps the induction going and proves the claim.
\end{subproof}

\begin{claim}[Claim 2]
 For all $k\ge 1$, the semicircle $R_{m+k+1}$ surrounds 
  $\Csep{m+k}{k}$, $\Csepp{m+k}{k}$ and all points in $T_{s_k}$ with real 
  parts at most $\rho_{m+k+1}/2$.
\end{claim}
\begin{subproof} Recall that $\Csep{m+k}{k}$ and $\Csepp{m+k}{k}$ have
 real parts at most $\rho_{m+k+1}/2$ by Condition~(\ref{Item:RealParts}).

 So suppose that $z\in T_{s_k}$ has $\re z \leq \rho_{m+k+1}/2$.
  We have $|\im w_{k}| \le |w_{k}| < \sep \rho_{m+k}$ by the first claim, 
  and since $T_{s_{k}}$ contains $w_{k}$ as well as $z$
  and has height at most $H$ (Condition~(\ref{Item:Domain})),
  it follows that $|\im z| < \sep \rho _{m+k}+H$. So, by
  Condition~(\ref{Item:RhoEstimates}),  
  \[  |z| \le \re z + |\im z| < 
       \rho_{m+k+1}/2 + \sep\rho_{m+k}+H < \rho_{m+k+1}
        \qedhere\]
\end{subproof}


Now suppose there is a curve $\gamma\subset J_{\s}$ that converges to $\infty$, and suppose that $w_0$ was chosen with $w_0\in\gamma$. For every $k\ge 0$, the curve $F^{\circ k}(\gamma)$ connects $w_k$ to $\infty$ (Condition~(\ref{Item:Infinity})). The point $w_k$ is surrounded by both $R_{m+k+1}$ and $\sep R_{m+k+1}$: by the first claim, we have $|w_k|<\sep\rho_{m+k}<\rho_{m+k+1}<\sep\rho_{m+k+1}$. As a result, $F^{\circ k}(\gamma)$ must contain a subcurve $\gamma_{k}$ connecting $R_{m+k+1}$ with $\sep R_{m+k+1}$. But this implies that $F^{\circ(k-1)}(\gamma)$ contains a subcurve $\gamma_{k-1}$ connecting $\Csep{m+k}{k-1}$ with $\Csepp{m+k}{k-1}$ (Condition~\ref{Item:Geodesics}). Since $\Csep{m+k}{k-1}$ and $\Csepp{m+k}{k-1}$ have real parts greater than $\sep \rho_{m+k}+H$ by Condition~(\ref{Item:RealParts}), it follows that both endpoints of $\gamma_{k-1}$ are outside of $\sep R_{m+k}$.
But $\gamma_{k-1}$ must be contained within $T_{s_{k-1}}$,  so Condition~(\ref{Item:TurnPoints}) implies that $\gamma_{k-1}$ must contain a point $z_{k-1}\in T_{s_{k-1}}$ with real part $\rho_{m+k}/2$. Now the last claim shows that $z_{k-1}$ is surrounded by $R_{m+k}$. As a result, $\gamma_{k-1}$ must contain two disjoint subcurves that connect $R_{m+k}$ with $\sep R_{m+k}$.

Continuing the argument inductively, it follows that $\gamma$ contains $2^k$ disjoint subcurves that connect $\Csep{m+1}{0}$ with $\Csepp{m+1}{0}$. Since this is true for every $k\ge 0$, this is a contradiction.
\end{proof}

Now we give conditions under which the set $J$ not only contains no curve to 
 $\infty$, but in fact no unbounded curve at all. 
 In many cases these conditions are satisfied automatically, such as in the 
 example that we construct below (see Theorem~\ref{Thm:RealizationTracts}).

\begin{cor}[Bounded Path Components]
\label{cor:boundedpathcomponents}
Suppose that, in addition to the conditions of 
 Theorem~\ref{Thm:NoCurveToInfinity}, there are countably many disjoint  hyperbolic geodesics $\sepp{C}_k\subset T$ so that all $F_0(\sepp{ C}_k)$ are semi-circles in $\H$ centered at $0$ with radii $\sepp{\rho}_{k+1}>\sep\rho_{k+1}$ so that the bounded component of $T\sm\sepp{ C}_{k+1}$ has real parts at most 
 $\sepp{\rho}_{k+1}/2$.

Then every path component of $J$ is bounded.
\end{cor}
\begin{proof}
We continue the proof of the previous theorem. Suppose there is an unbounded 
  curve $\gamma\subset J_{\s}$ with $w_0\in\gamma$. For every 
  $k\ge 0$ the curve $F^{\circ k}(\gamma)$ connects $w_k$ to points at 
  arbitrarily large real parts. We will show that there must be a point 
  $z_0\in\gamma$ so that for every $k$ the subcurve of $\gamma$ between $w_0$ 
  and $z_0$ contains $2^k$ disjoint subcurves that connect $\Csep{m+1}{0}$ 
  with $\Csepp{m+1}{0}$, and this is a contradiction.

Let $\sepp{ R}_k$ be semi-circles centered at $0$ with radii 
 $\sepp{\rho}_{k}$. Since $\sepp{\rho}_{k+1}>\sep\rho_{k+1}$, it follows that 
 every $\sepp{ C}_k$ is in the unbounded component of $T\sm\sep C_k$. 
 Define vertical translates $\Cseppp{m}{k}=\sepp{ C}_{m}+2\pi is_k$ in 
 analogy to the $\Csep{m}{k}$ and $\Csepp{m}{k}$. 
 As in the second claim in the proof above, it follows that the bounded 
 component of $T_{\s_k}\setminus \sepp{ C}^k_{m+k+1}$ 
 is surrounded by $\sepp{ R}_{m+k+1}$. 

By the first claim in the proof above, $\Csepp{m}{0}$ separates $w_0$ from $\infty$ within $T_{s_0}$, so $\Cseppp{m}{0}$ and also $\Cseppp{m+1}{0}$ must do the same. Let $z_0$ be a point in the intersection of $\gamma$ with $\Cseppp{m+1}{0}$ and denote by $[w_0,z_0]_\gamma$ the subcurve of $\gamma$ connecting $w_0$ with $z_0$. 
Then $F([w_0,z_0]_\gamma)$ connects $w_1$ with $F(z_0)\in \sepp{ R}_{m+2}$.
So $F(z_0)$ belongs to the unbounded component of  
 $T_{s_1}\sm \Cseppp{m+2}{1}$, and 
 there is thus a point $z_1\in[w_0,z_0]_\gamma$ with $F(z_1)\in\Cseppp{m+2}{1}$, and $F^{\circ 2}([w_0,z_1]_\gamma)$ connects $w_2$ with $\sepp{ R}_{m+3}$. 
\hide{
Similarly, there is a point $z_2\in[w_0,z_1]_\gamma$ with $F^{\circ 2}(z_2)\in\Cseppp{m+3}{2}$, so that $F^{\circ 3}([w_0,z_2]_\gamma)$ connects $w_2$ with $\sepp{ R}_{m+4}$ and so on:
} 
By induction, for any $k\ge 0$, the curve $F^{\circ k}([w_0,z_{k-1}]_\gamma)$ connects $w_k$ with $\sepp{ R}_{m+k+1}$ and hence it connects $R_{m+k+1}$ with $\sep R_{m+k+1}$.

The same arguments from the proof of the theorem now show that in fact $[w_0,z_k]_\gamma\subset[w_0,z_0]_\gamma$ must contain $2^k$ subcurves connecting $\Csep{m+1}{0}$ with $\Csepp{m+1}{0}$ for every $k\ge 0$, and this is the desired contradiction.
\end{proof}

\begin{remark} 
 We stated the results in the form above in order to minimize the order 
  of growth of the resulting entire function, and to show that the entire 
  functions we construct can be rather close to finite order; 
  see Section~\ref{sec:properties}. If we were only interested in the 
  non-existence of unbounded path components in $I$, we could 
  have formulated  conditions that are somewhat simpler than those in the 
  preceding theorem and its corollary. For instance, the necessity for 
  introducing a third geodesic $\sepp{ C}_k$ would have been removed if 
  we had placed $C_k$ at 
  real 
  parts at most $\rho_k/2$, and required that 
  the entire bounded component of 
  $T\sm C_k$ has real parts less than $\rho_k/2$ 
  (the image of 
   any curve in $T$ connecting $\Csepp{m+k}{k}$ with $\Csep{m+k+2}{k}$ would 
   then connect $\Csepp{m+k+1}{k+1}$ and $\Csep{m+k+3}{k+1}$; this keeps 
   the induction going as before.) 
\end{remark}

 \begin{thm}[Tract with Bounded Path Components]
  \label{Thm:RealizationTracts}
   There exist a tract $T$ with $\cl{T}\subset\H$ and a conformal isomorphism
    $F_0:T\to\H$ fixing $\infty$ that satisfies the conditions of
    Theorem \ref{Thm:NoCurveToInfinity} and Corollary~\ref{cor:boundedpathcomponents}, so every path component of $J$ is bounded. 
    
    In fact, $T$ and $F_0$ can be chosen so as to
    satisfy the following conditions 
    for an arbitrary $M\in(1,1.75)$ (with the same notation as in Theorem 
     \ref{Thm:NoCurveToInfinity}):
  \begin{enumerate}
   \item[(\ref{Item:RhoEstimates}')]
    $\rho_k^{M}<\sep{\rho}_k$ and
    $\sep{\rho}_k^{M}<\rho_{k+1}$;
   \item[(\ref{Item:RealParts}')] 
     the geodesics $C_k$ and $\sep{C}_k$ have real parts strictly
      between $\sep{\rho}_k^{M}$ and $\rho_{k+1}/3$;
   \item[(\ref{Item:NotFarBack}')] 
     all points in the unbounded component of $T\sm \sep{C}_k$ have
      real parts greater than $\sep{\rho}_k^{M}$;
   \item[(\ref{Item:TurnPoints}')] every curve in $T$ that connects $C_k$ to $\sep{C}_k$
      intersects the line $\{z\in\C:\re(z)=\rho_k^{1/M} \}$;
      \item[(h')] the geodesics $\sepp{ C}_{k+1}$ from Corollary~\ref{cor:boundedpathcomponents} have the property that the bounded component of $T\sm\sepp{ C}_{k+1}$ has real parts at most $(\sepp{\rho}_{k+1})^{1/M}$.
  \end{enumerate}
 \end{thm}

\begin{remark}
The modified conditions as written in this theorem are needed in order to show 
 that this tract is ``approximately'' realized by an entire function: 
 they
 are adapted to the quality of the approximation 
 that we get later in this section. 
\end{remark}

\begin{proof}


Our domain $T$ will be a countable union of long horizontal tubes of unit thickness, together with countably many vertical tubes and countably many turns made of quarter and half annuli, all of unit thickness as well (see Figure~\ref{FigWiggleConstruction}). The domain $T$ terminates at the far left with a semidisk at center $P$. The lengths of the various tubes are labelled as in Figure~\ref{FigWiggleConstructionDetails}.

\begin{figure}
\begin{center}
\subfigure[\label{FigWiggleConstruction}The tract $T$.]{%
\setlength{\unitlength}{1cm}
\begin{picture}(13,4)
\put(-2,0){\includegraphics[clip,width=160mm]{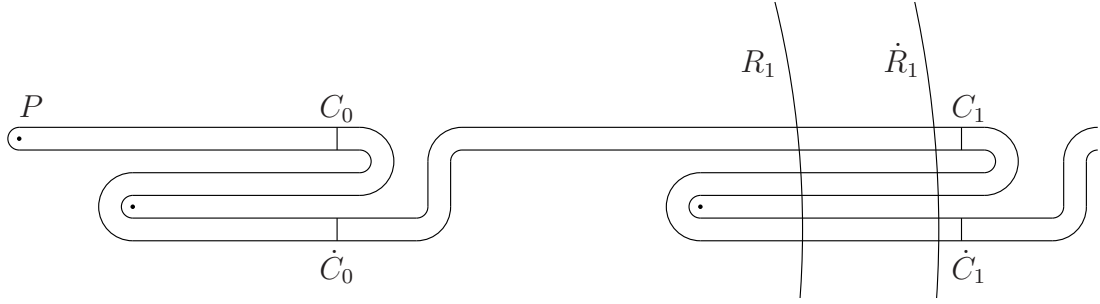}}
\put(-0.5,2.4){$P$}
\put(3.5,0.25){$\sep C_0$}
\put(3.5,2.4){$C_0$}
\put(9.1,3){$R_1$}
\put(11.0,3){$\sep R_1$}
\put(11.9,2.4){$C_1$}
\put(11.9,0.25){$\sep C_1$}
\end{picture}}
\end{center} 
\begin{center}
\subfigure[\label{FigWiggleConstructionDetails}Length scales in the construction.]{%
\setlength{\unitlength}{1cm}
\begin{picture}(12,11)
\put(0,0){\includegraphics[width=120mm]{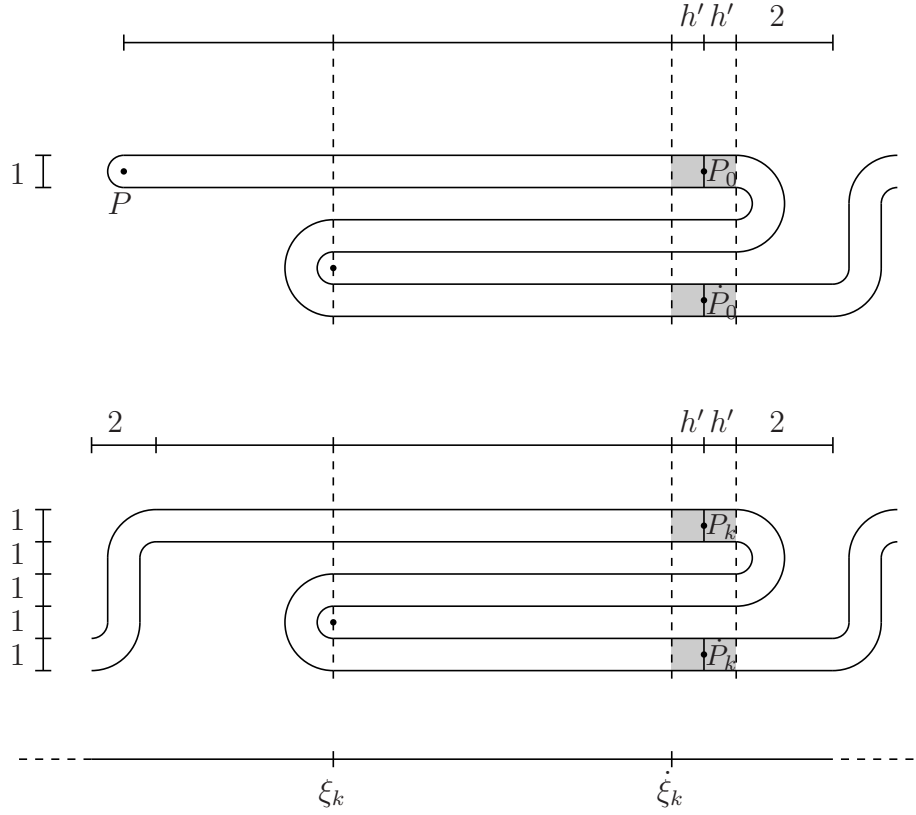}}
\put(4.3,0.2){$\xi_k$}
\put(8.8,0.2){$\sep \xi_k$}
\put(9.45,2){\small $\sep P_k$}
\put(9.45,3.74){\small $P_k$}
\put(0.2,2.0){$1$}
\put(0.2,2.45){$1$}
\put(0.2,2.9){$1$}
\put(0.2,3.35){$1$}
\put(0.2,3.8){$1$}
\put(1.5,5.1){$2$}
\put(9.1,5.1){$h'$}
\put(9.5,5.1){$h'$}
\put(10.3,5.1){$2$}
\put(1.5,8){$P$}
\put(0.2,8.4){$1$}
\put(9.45,6.67){$\sep P_0$}
\put(9.45,8.44){$P_0$}
\put(9.1,10.5){$h'$}
\put(9.5,10.5){$h'$}
\put(10.3,10.5){$2$}
\end{picture}}
\end{center}
\caption{Construction of an example for Theorem~\ref{Thm:RealizationTracts}.
 The length $h'$ is indepednent of $k$. In (b), the boxes
  $Q_k$ and $\sep Q_k$ are shaded.}
\end{figure}

More precisely, our tract $T$ is specified by the length $h'\geq 1$
  (fixed below) and sequences
 $(\xi_k)_{k\geq 0}$ and $(\sep\xi_k )_{k\geq 0}$, with 
 $\xi_0>2$ and $\xi_k < \sep\xi_k < \xi_{k+1}- 4 - 2h'$
 for all $k$. Let us set $P := 1$, $P_k := \sep\xi_k+h'$ and
 $\sep P_k := P_k - 4i$. We define a curve 
   \[ \Gamma=\bigcup_{k\geq 0}\gamma_k \cup \sep\gamma_k, \]
 where $\gamma_0$ is the straight line segment $[P,P_0]$ and 
   \begin{align*}
     \sep\gamma_k = &[P_k,P_k+h']\, \cup\,
          \{P_k + h'-i+e^{2\pi i \theta}:|\theta|<\pi/2\}\, \cup \\
             &[P_k + h' - 2i , \xi_k - 2i ]\, \cup\,
          \{ \xi_k - 3i + e^{2\pi i\theta}: |\theta-\pi|<\pi/2 \}\, \cup \\
            &[\xi_k - 4i , \sep P_k ]; \\
     \gamma_{k+1} = &[ \sep P_k , \sep P_k + h' + 2  ]\, \cup
          \{ \sep P_k + h' + 2 + i + e^{2\pi i \theta} :
             \theta\in (-\pi/2,0)\}\, \cup \\
            &[\sep \xi_k + 2h' + 3 - 2i ,
               \sep \xi_k + 2h' + 3 - i ]\, \cup\, \\
          &\{ \sep \xi_k + 2h' + 4 - i + e^{2\pi i \theta} :
                  \theta\in (\pi/2,\pi) \}\, \cup \\
               &[\sep \xi_k + 2h' + 4, P_{k+1} ]. 
    \end{align*}
  The tract $T$ then consists of all points that have distance at most 
   $1/2$ from $\gamma$. The map $F_0$ is chosen such that
   $F_0(1)=1$ and $F_0(\infty)=\infty$; 
   this detemines $F_0$ completely.

Note that $T$ and $F_0$ 
  (regardless of the choices of $\xi_k$, $\sep\xi_k$ and $h'$)
  satisfy conditions 
  (\ref{Item:Infinity}) and (\ref{Item:Domain}), 
  where $H=5<2\pi$. 
  We set $\rho_{k+1}:=|F_0(P_k)|$ and 
   $\sep \rho_{k+1}:=|F_0(\sep{P}_k)|$. 
  Let $R_{k+1}$ and $\sep R_{k+1}$ be the semicircles around $0$ with radii 
  $\rho_{k+1}$ and $\sep \rho_{k+1}$ and let $C_k:=F_0^{-1}(R_{k+1})$ and 
  $\sep C_k:=F_0^{-1}(\sep R_{k+1})$. 

Then $C_k$ and $\sep C_k$ are hyperbolic geodesics of $T$. 
 If $h'$ is sufficiently large
  -- in fact, $h':=2$ is sufficient, see Lemma 
  \ref{Lem:GeometryGeodesic} in the appendix -- 
   then $C_k$ and $\sep C_k$ will be contained
 in the boxes
 $Q_k:=\{z\in\C\colon \re(z)\in (x_k^M,x_k^M+2h'), |\im z|<1/2\}$ and 
 $\sep{Q}_k:=Q_k-4hi$, and they connect the upper with the lower boundaries of 
 their boxes. 
 In particular,
 $\sep C_k$ separates $C_k$ from $C_{k+1}$, so 
 condition~(\ref{Item:Geodesics}) is also satisfied.

We now define the sequences $\xi_k$ and $\sep\xi_k$. Begin by choosing
   $\xi_0 > 2$ sufficiently large (see below) and setting
   $\sep\xi_0 := \xi_0^{12M^2}$. We then proceed inductively by setting
\[
\xi_{k+1}:=\exp\left(\sep \xi_k/M\right)
\qquad\mbox{ and }\qquad
\sep \xi_{k+1}:=\exp\left(12M\sep \xi_k\right)= \xi_{k+1}^{12M^2}
\,\,.
\]

To see that these indeed give rise to a well-defined tract $T$
 as above, we need to verify that 
  \begin{equation}  \label{eqn:intermediatetubes}
   \xi_{k+1} = \exp\left(\sep\xi_k/M\right) > \sep\xi_k + 4 + 2h'.
  \end{equation}
If $\xi_0$ --- and hence all $\xi_k$ --- was chosen sufficiently large,
 then this inequality will certainly hold. 
 We will use other, similar, elementary inequalities below 
 that may hold only if $\xi_0$ is  
 sufficiently large. We use the symbol
 ``$\star$'' to mark such inequalities (e.g.\
 ``$\exp(\sep\xi_k/M)\stackrel{\star}{>} \sep\xi_k + 4+2h'$'').

 It is easy to see that the remaining conditions from
 Theorem \ref{Thm:NoCurveToInfinity} and 
 Corollary~\ref{cor:boundedpathcomponents}
  follow from the modified ones
  in the statement of the theorem, provided that $\xi_0$ was chosen
  sufficiently large. So it remains to verify 
  (\ref{Item:RhoEstimates}')
   to (h'); 
  we begin by estimating $\rho_{k+1}$ and $\sep\rho_{k+1}$ for $k\geq 0$. 
  We claim that
  \begin{align}
    \xi_{k+1}^M = \exp(\sep \xi_k )<
       \rho_{k+1}<&\exp\left((4\pi/3)\,\sep \xi_k\right)\quad
      \text{and} \label{eqn:rhokestimate} \\
     \sep\rho_{k+1} <& \exp(12\sep \xi_k) = \sep \xi_{k+1}^{1/M}.
            \label{eqn:seprhokestimate}  
  \end{align}

  We prove the inequalities (\ref{eqn:rhokestimate}) and 
   (\ref{eqn:seprhokestimate}) using the hyperbolic metric in the domain $T$. 
   Indeed, we have $\log\rho_{k+1}=\dist_{\H}(P,R_{k+1})=\dist_T(P,C_k)$, and
   similarly for $\sep\rho_{k+1}$. Hence it suffices to estimate the
   hyperbolic distance between $P$ and $C_k$, which we can easily do
   using the standard estimate (\ref{eqn:standardestimate}). 

  Let $\gamma$ be the piece of $\Gamma$ that connects $P$ to
   $P_k$; i.e.\
    \[ \gamma=\bigcup_{j\leq k}\gamma_j\,\cup\,\bigcup_{j<k}\sep\gamma_j. \]
   If $k\geq 1$, we clearly have
    \[ \ell(\gamma) < \sep\xi_k + h' + 2(\sep\xi_{k-1}+ 2h') + 
           3k\pi\stackrel{\star}{<} \sep\xi_k + 3\sep\xi_{k-1}
           = \sep\xi_k + \log\sep\xi_k)/4M 
           \stackrel{\star}{<} (\pi/3)\sep\xi_k. \]
   For $k=0$, we also have $\ell(\gamma)=\sep\xi_0 - 1 + h'
      \stackrel{\star}{<} (\pi/3)\sep\xi_0$. So  
    \[ \log\rho_{k+1} = \dist_T(P,C_k) \leq \ell_T(\gamma) \leq
         4\ell(\gamma) < 
         (4\pi/3) \sep \xi_k. \]

   The upper bound for $\sep\rho_{k+1}$ is proved analogously.
    Let $\sep \gamma$ be the piece of $\Gamma$ connecting $P$ to
    $\sep P_k$. If $k\geq 1$, then
   \[
       \ell(\sep \gamma)  < 
       3(\sep\xi_k+2h') + 3(k+1)\pi -
        (\xi_k - \sep \xi_{k-1} - 2h').  \]
  Note that
    \[ \xi_k = \exp(\sep\xi_{k-1}/M) \stackrel{\star}{>}
       2\sep \xi_{k-1} +8h' \stackrel{\star}{>} 
       \sep \xi_{k-1} + 8h' + 3(k+1)\pi, \]
   so we have $\ell(\sep\gamma) < 3\sep\xi_k$. If $k=0$, also 
    \[ \ell(\sep\gamma) < 3(\sep \xi_0 + 2h') + \pi - 2\xi_0 
           \stackrel{\star}{<} 3\sep\xi_0. \]
  Hence $\log\sep\rho_{k+1} \leq 4\ell(\sep\gamma)< 12\sep\xi_k$.
 
 To prove the
    lower bound for $\rho_{k+1}$, note that \emph{any} curve
    $\alpha$ connecting $P$ to $C_k$ must have 
    $\ell(\alpha) \geq \sep \xi_k + h' - 1 \geq \sep \xi_k$, and
    every point of $\alpha$ has distance at most $1/2$ from $\partial T$.
    Therefore
  \[ \log \rho_{k+1} \geq \inf_{\alpha} \ell_T(\alpha) \geq
     \inf_{\alpha} \ell(\alpha) \geq \sep \xi_k, \]
   as claimed.

Now we show that $\rho_{k+1}$ and $\sep\rho_{k+1}$ satisfy condition 
  (\ref{Item:RhoEstimates}'). 
  The second inequality follows from (\ref{eqn:intermediatetubes}),
   (\ref{eqn:rhokestimate}) and (\ref{eqn:seprhokestimate}):
   \[ \sep \rho_{k}^M < \sep \xi_k <  \xi_{k+1} < \rho_{k+1}^{1/M} < \rho_{k+1}. \] 
 To prove the first inequality, note that 
  the subdomain of $T$ bounded by $C_k$ and 
  $\sep C_k$ maps under $F_0$ conformally onto the semi-annulus in $\H$ between 
  the semicircles $R_{k+1}$ and $\sep R_{k+1}$. 
  So the moduli are equal, and we see by the Gr\"otzsch inequality that
\[
\frac{1}{\pi} \log(\sep\rho_{k+1}/\rho_{k+1}) > 
   2\left(\sep \xi_k - \xi_k \right)
  = 2\left(\sep \xi_k - \sep \xi_k^{1/(12M^2)} \right) 
    \stackrel{\star}{>} \sep \xi_k 
\]
and so, using (\ref{eqn:rhokestimate}),
\[
\sep\rho_{k+1}>\rho_{k+1}\exp\left(\pi \sep \xi_k \right) 
= \rho_{k+1} \left(\exp\left( (4\pi/3)\sep \xi_k \right)\right)^{3/4}
> \rho_{k+1}^{1.75}
> \rho_{k+1}^M
\,\,.
\]

The construction of 
  $C_{k+1}$ and $\sep C_{k+1}$ is such that their real parts are at least 
   $\sep \xi_{k+1}$, which is larger than  $\sep \rho_{k+1}^M$
     by
   (\ref{eqn:seprhokestimate}), and at most 
  \[ 
  \sep \xi_{k+1} + 2h' =
     (\log \xi_{k+2})/M + 2h'< \stackrel{\star}{<} \xi_{k+2}/3 <
     \rho_{k+2}/3 \]
  (using (\ref{eqn:rhokestimate})),  
   so condition~(\ref{Item:RealParts}') is satisfied.

Condition~(\ref{Item:NotFarBack}') is obvious: the construction ensures that 
  all points in the unbounded component of $T\sm \sep C_{k+1}$ have real parts 
  at least $\sep \xi_{k+1}$, which is
   greater than $\sep \rho_{k+1}^M$ by (\ref{eqn:seprhokestimate}).

Furthermore, every curve in $T$ that connects $C_{k+1}$ with $\sep C_{k+1}$ must 
  reach real parts less than $\xi_k < \rho_k^{1/M}$, and this is
  condition~(\ref{Item:TurnPoints}').

To conclude, we show that (h') is satisfied, so that 
 Corollary~\ref{cor:boundedpathcomponents} applies and 
 all path components of $J$ are bounded.
 We define $\sepp P_k := M(12M+1)\sep\xi_k$. Since
  $M(12M+1)\sep\xi_k \stackrel{\star}{<} \exp(\sep\xi_k/M)=\xi_{k+1}$, 
  we have $P_k\in T$, so we can set 
  $\sepp\rho_k := |F_0(\sepp P_k)|$. Let 
 $\sepp{ R}_k$ be the semicircles in $\H$ centered at $0$ with radii 
 $\sepp{\rho}_k$, and let $\sepp{ C}_k\ni\sepp P_k$ be the $F_0$-preimage of 
 $\sepp{ R}_{k+1}$ within $T$. Then, using 
 Lemma~\ref{Lem:GeometryGeodesic} as above,
 all points in the bounded component of $T\setminus \sepp C_{k+1}$ have
 real parts at most $\re \sepp P_{k+1} + h'$. 

 We can again use the hyperbolic metric to estimate
  $\log\sepp \rho_{k+1} = \dist_T(P,\sepp C_k) > \re \sepp P_k$. Hence
  \begin{align*}
    \sepp \rho_{k+1}^{1/M} > \exp((12M+1)\sep\xi_{k}) &=
     \exp(\sep\xi_k)\cdot \sep \xi_{k+1}  \\ &\stackrel{\star}{>} 
      M(12M+1)\sep\xi_{k+1} + h'
      =
      \re \sepp P_{k+1} + h'. \qedhere \end{align*}
\end{proof}

 In order to complete the proof of Theorem \ref{thm:counterexample}, we
  need to show that there is an entire function that suitably
  approximates
  the previously constructed map. To this end, we will use the
  following fact on the existence of entire functions with a
  prescribed tract; a proof can be found in the next section.
  Let us introduce the following notation: if $F:T\to\H$ is a conformal
  isomorphism, then a geodesic of $T$ that is mapped by $F$ to a semicircle
  centered at $0$ is called a \emph{vertical geodesic} (of $F$).

  \begin{prop}[Approximation by entire functions]
   \label{prop:approximation2}
   Let $T$ be a tract, and let 
    $F:T\to\H$ be a conformal isomorphism fixing $\infty$. Let $\theta>1$.

Then there is an entire function $g\in\B$ with 
 $S(f)\subset B_1(0)$ 
 and a single tract 
  $W = g^{-1}(\{|z|>1\})$, 
  and a $2\pi i$-periodic
  logarithmic transform $G:\log W \to \H$ 
  of $g$ with the following properties:
   \begin{enumerate}
   \item
   $\log W$ has a component $\wt T$ satisfying $\wt T\subset T$;
    \item the vertical geodesics of $G$ have uniformly bounded
     diameters;
    \item 
      $\displaystyle{|F(z)| \leq |G(z)| \leq |F(z)|^{\theta}}$ 
     when $z\in\wt{T}$ with $\re z$ sufficiently large.
         \label{item:sizeofG}
  \end{enumerate}
 \end{prop}
\begin{remark} 
  If we apply the above proposition to a tract $T$ with 
    $\cl{T}\subset\H$ (such as the one from Theorem 
    \ref{Thm:RealizationTracts}), then the resulting 
    function $g$ satisfies 
    $f(B_1(0))\Subset B_1(0)$. 
    It follows that the postsingular set is compactly 
    contained 
    in the Fatou set of $g$, and hence that $g$ is hyperbolic. 
\end{remark} 
 \begin{proof}[Proof of Theorem \ref{thm:counterexample} (using Proposition
  \ref{prop:approximation2})]
 Let $F_0\in\Blog$ be the function con\-struc\-ted in
  Theorem \ref{Thm:RealizationTracts}, and let $T$ be its single tract. 
  Choose $1<\theta<M$ (where $M$ is the constant from Theorem 
  \ref{Thm:RealizationTracts}).
  Let $g$ be a function as in Proposition \ref{prop:approximation2}, with
   logarithmic transform $G:\wt{T}\to\H$. 
   (Recall that $G$ extends continuously to the
   closure $\operatorname{cl}(\wt{T})$.)

  The vertical geodesics $C_k$ and $\sep{C}_k$ of
   $T$ intersect $\wt{T}$ for
   sufficiently large $k$.
   Let $\sep{\sigma}_{k+1}$ be maximal with the property that
   the geodesic 
   $\sep{D}_k := \{z\in\operatorname{cl}(\wt{T}): |G(z)|=\sep{\sigma}_{k+1}\}$
   intersects $\sep{C}_k$, and define $\sigma_{k+1}$ and $D_k$ similarly.
   We claim that, with this choice of geodesics,
   the function $G$ also satisfies the conclusions of Theorem 
   \ref{Thm:RealizationTracts} (for a constant $M'<M/\theta$). 

  Indeed, by (\ref{item:sizeofG}) of Proposition
    \ref{prop:approximation2}, 
    we have $\rho_k\leq \sigma_k \leq \rho_k^{\theta}$ and
    $\sep{\rho}_k\leq \sep{\sigma}_k\leq \sep{\rho}_k^{\theta}$. 
    Thus
    \[ \sigma_k^{M'} \leq \rho_k^M < \sep{\rho}_k \leq \sep{\sigma}_k
       \quad\text{and}\quad
       \sep{\sigma}_k^{M'} \leq
        \sep{\rho}_k^M < \rho_{k+1} \leq \sigma_k.
     \]
  Thus
   (\ref{Item:RhoEstimates}') holds. (\ref{Item:RealParts}')  
   and (\ref{Item:TurnPoints}') follow similarly, using the fact that the
   geodesics $D_k$ and $\sep{D}_k$
   have uniformly bounded diameters. 

 To prove (\ref{Item:NotFarBack}'), note that the unbounded component of
  $\wt{T}\sm \sep{D}_k$ does not intersect $\sep{C}_k$
  since we chose $\sep{\sigma}_{k+1}$ to be maximal. Hence it
  follows from condition (\ref{Item:NotFarBack}') for $F_0$ 
  that this component
  has real parts at least $\sep{\rho}_k^M \geq \sep{\sigma}_k^{M'}$. 
  Property (h') follows analogously.
\end{proof}

\section{Entire functions with prescribed tracts}
\label{sec:realization}

  We will now prove Proposition \ref{prop:approximation2}.
   Eremenko and Lyubich \cite{alexmishaexamples} were the
   first to introduce methods of approximation theory into 
   holomorphic dynamics; more precisely they used Arakelian's approximation
   theorem to construct various entire functions with ``pathological''
   dynamics. It is possible to likewise use this theorem to approximate
   any given tract by a logarithmic tract of an entire function; this 
   would be enough to give a counterexample to the strong form of
   Eremenko's conjecture. However, Arakelian's theorem provides no 
   information on the singular values of the approximating map, so a function
   obtained in this manner might not belong to the Eremenko-Lyubich class.

  Hence we instead
   use the method of
   approximating a given tract using Cauchy integrals,
   which is also well-established. 
   (Compare e.g.\ \cite{eremenkogoldberg} for a similar construction.)
   There appears to be no
   result stated in the literature that is immediately 
   applicable to our situation. We will therefore first provide
   a proof of the following, more classical-looking statement, and
   then proceed to indicate how it implies Proposition 
   \ref{prop:approximation2}.

\begin{prop}[Existence of functions with prescribed tracts] 
\label{prop:approximation}
   Let $V\subset\C$ be an unbounded Jordan domain and let
   $\Psi:V\to \H$
   be a conformal isomorphism with $\Psi(\infty)=\infty$.  
   Let $\rho$ be arbitrary with $1<\rho<2$ and define
    \[ f:V\to\C;z\mapsto \exp\bigl((\Psi(z))^{\rho}\bigr). \]
   Then there exists an entire function $g\in\B$ and a number
    $K>0$ such that 
    the following hold:
    \begin{enumerate}
     \item
      $W := \{z:|g(z)|>K\}$ is a simply connected domain
      that is contained in $V$, 
      and $g|_{W}$ is a universal covering;
     \item $|g(z) - f(z)| = O(1)$ on $W$, and
        $g(z)=O(1)$ outside $W$. 
    \end{enumerate}
  \end{prop}
 \begin{remark}
  In particular, the tract $W$ of $g$ satisfies
   \[ V\supset W \supset \{z: \re \Psi(z) > C \text{ and }
                              |\arg \Psi(z)| < \theta \} \]
  (where $\theta$ can be chosen arbitrarily close to 
    $\pi/2\rho$ if $C$ is sufficiently large). So this proposition really
    does present a result on the realization of a prescribed tract
    (up to a certain ``pruning'' of the edges) by an entire function. 
 \end{remark}
  \begin{proof} The idea of the proof is simple: we define a function
   $h$, using an integral along the boundary $\alpha$ of the
   desired tract, which changes by 
    $f(z)$ as $z$ crosses the curve $\alpha$. That is, we set 
    \[ h(z) := \frac{1}{2\pi i}\int_{\alpha} \frac{f(\zeta)}{\zeta - z}d\zeta.
       \] 
   We are using $\Psi^{\rho}$, rather than
    $\Psi$ itself, in the definition of $f$ 
    to ensure that this integral converges
    uniformly and that $h$ is bounded.
    Then the function $g$ that agrees with $h$ on the outside
    of $\alpha$ and with $h+f$ on the inside will be entire, and it follows
    easily that it is in class $\B$.

   Let us now provide the details of this argument.
    We define
    $\Phi := \Psi^{\rho}$ and let $S$ denote the sector
    $S := \Phi(V) = \bigl\{z:|\arg z|<\frac{\rho\pi}{2}\bigr\}$. 
    Also fix some $\eta\in (\pi/2,\rho\pi/2)$ and set
     $\nu := \exp(i\eta)$. We define
     \begin{align*}
       \alphat &: (-\infty,\infty)\to S; 
           t\mapsto \begin{cases}
                      1 + \nu t & t\geq 0 \\
                      1 + \overline{\nu} |t| & t < 0 
                    \end{cases} \quad \text{and} \quad
       \alpha := \Phi^{-1}\circ\alphat. 
    \end{align*}
    Let $V'$ denote the component of $\C\setminus\alpha$ with
     $V'\subset V$.

\begin{figure}%
 \begin{center}
  \resizebox{!}{.3\textheight}{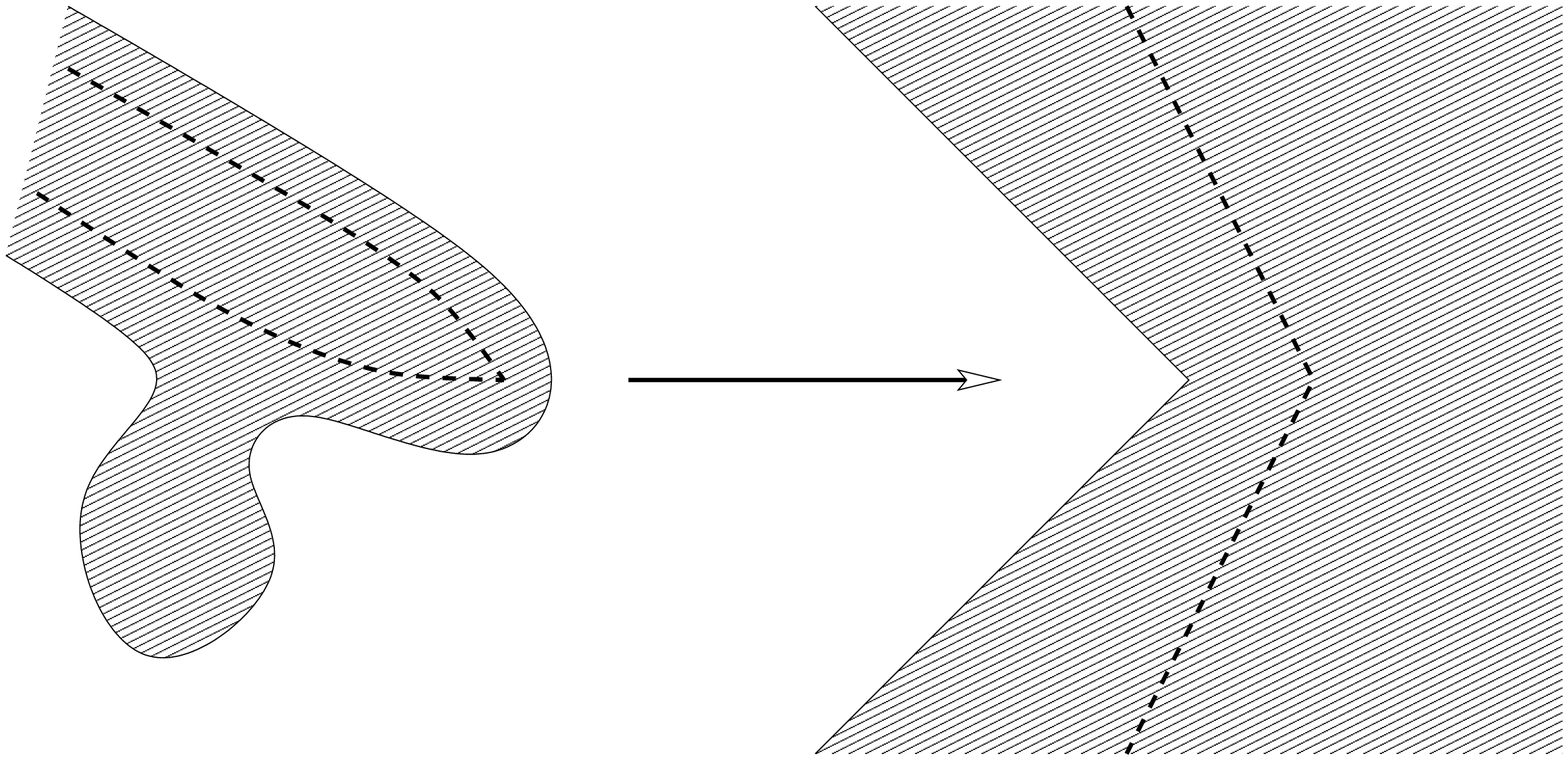}
 \end{center}
 \caption{Definition of $\alpha$ and $\alphat$ in the proof of
   Proposition  
   \ref{prop:approximation}}
\end{figure}

\begin{claim}[Claim 1]
   The integral 
    $\int_{\alpha} f(\zeta) d\zeta$
   converges absolutely. In particular, 
    \[ h(z) := \frac{1}{2\pi i}\int_{\alpha} \frac{f(\zeta)}{\zeta - z}d\zeta.
       \] 
   defines a holomorphic function $h:\C\setminus\alpha\to\C$. 
\end{claim}
  \begin{subproof}
   Note that $|\alpha'(t)|=|1/\Phi'(\alpha(t))|$ for $t\neq 0$. 
   By the Schwarz lemma and Koebe's theorem, we have
    \[ |\Phi'(\alpha(t))| \geq
         \frac{\dist(\alphat(t),\partial S)}{%
               4\dist(\alpha(t),\partial V)}. \]
   Clearly $\dist(\alphat(t),\partial S)\geq C_1(1+|t|)$
    for some $C_1 >0$. So the hyperbolic length of $\alpha|_{[0,t]}$ satisfies
    $\ell_S(\alpha|_{[0,t]})=O(\log (|t|+1))$. On the other hand, by 
    the standard estimate
    (\ref{eqn:standardestimate}), the density $\lambda_V$ of the hyperbolic
    metric on $V$ satisfies $\lambda_V(z)\geq 1/(2(|z|+|z_0|))$ for all
    points $z\in V$ (where $z_0$ is an arbitrary point of $\partial V$),
    which means that 
    $\log|z| = O(\dist_V(z,\alpha(0)))$ as $z\to\infty$.
    Combining these estimates, we see that 
    $|\alpha(t)|$ grows at most polynomially in $|t|$; in particular,
    $\dist(\alpha(t),\partial V) \leq C_2(1+|t|)^c$ for some $c,C_2>0$. 

   Together, these estimates imply by the Koebe theorem that 
    $|\Phi'(\alpha(t))|\geq C/4(1+|t|)^{c-1}$ for $C:= C_1/C_2$. 
     In particular,
    \begin{align*}
      \int_{\alpha} |f(\zeta) d\zeta| &=
      \int_{-\infty}^{+\infty} \exp(\re\alphat(t))|\alpha'(t)|dt \\
     &=
      \int_{-\infty}^{+\infty} \frac{\exp(1 - |\re(\nu)t|)}{|\Phi'(\alpha(t))|}dt
     \leq
      \frac{e}{C}\int_{-\infty}^{+\infty} (1+|t|)^{c-1}
        e^{-|\re(\nu)t|} dt < \infty.
   \end{align*}
     This completes the proof. \end{subproof}

\begin{figure}
  \resizebox{!}{.3\textheight}{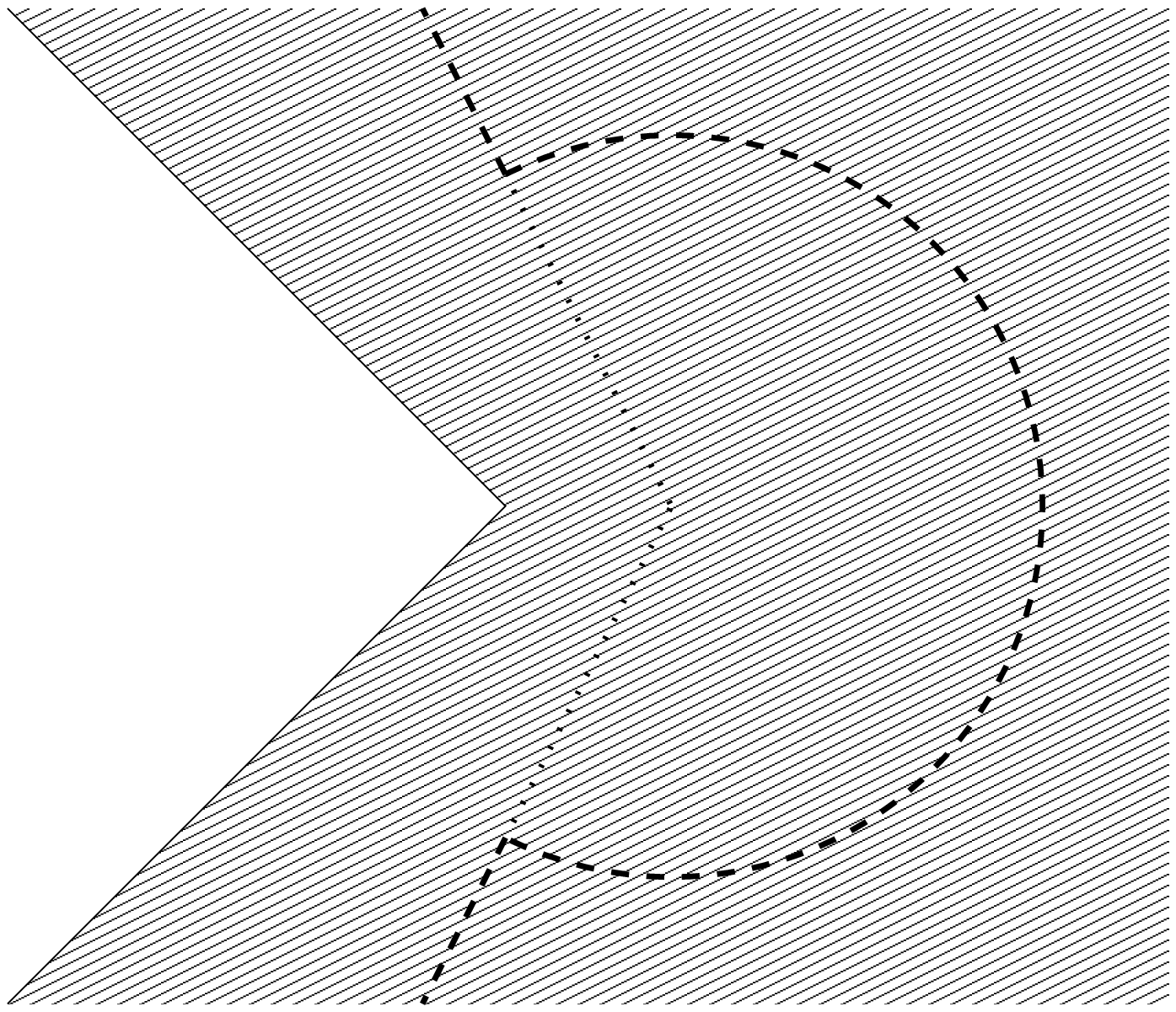}\hfill%
  \resizebox{!}{.3\textheight}{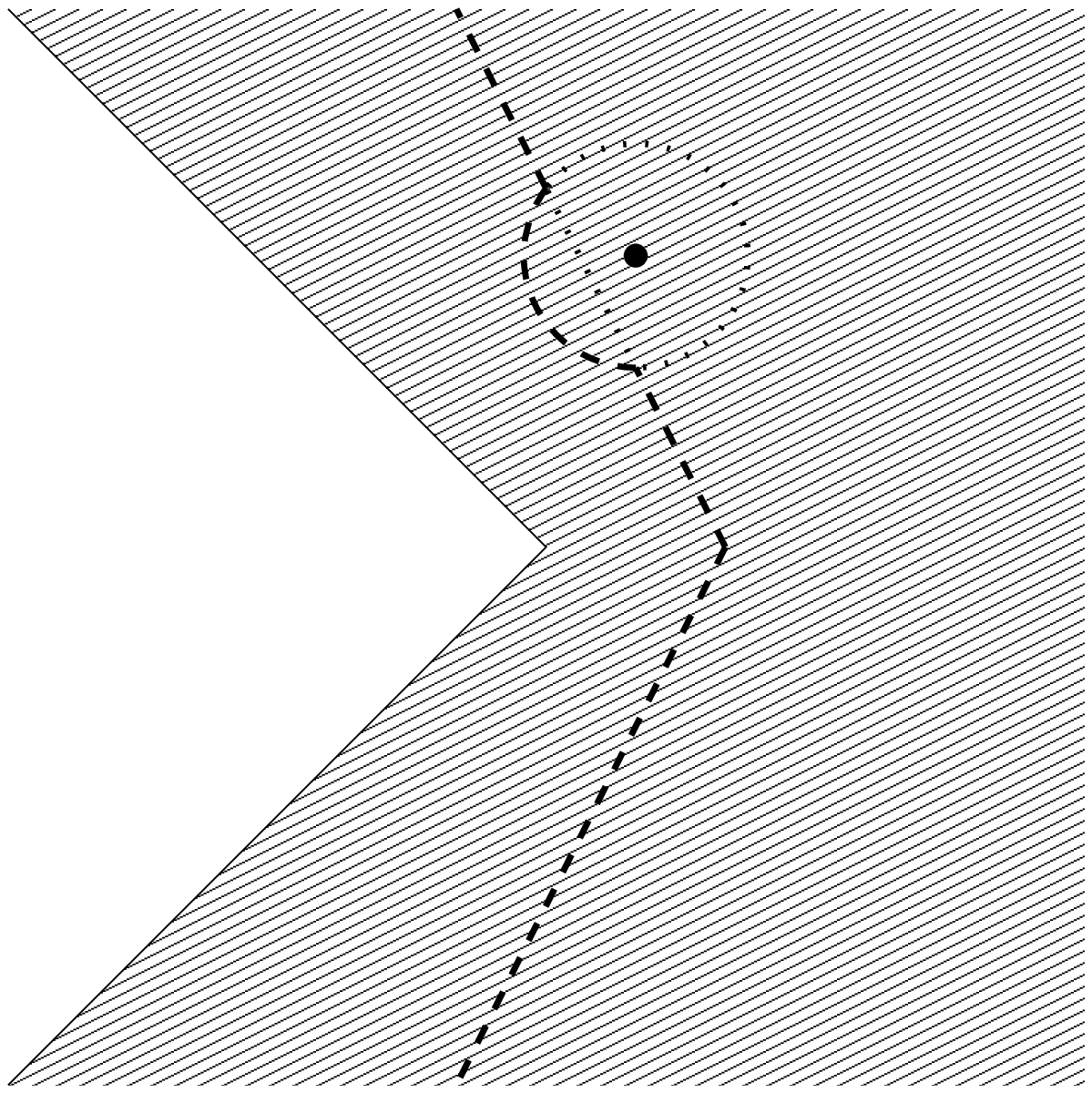}%
 \caption{Definition of $\betat$ and $\gammat$ as in Claims 2 and 3}
\end{figure}

\begin{claim}[Claim 2]
   The function 
    \[ g(z) := \begin{cases}
                 h(z) & z\notin \cl{V'} \\
                 h(z) + f(z)& z\in V'
               \end{cases} \]
   extends to an entire function $g:\C\to\C$.                 
\end{claim}
 \begin{subproof}
  Let $R\gg 1$ be arbitrary, and modify $\alphat$ to obtain a curve
    \[ \betat := (\alphat\cap \{|\zeta|>R\})\cup 
       \{1 + Re^{2\pi i \theta}:\theta \in [-\eta,\eta]\}. \]
  Set $\beta := \Phi^{-1}\circ \betat$ and let $W$ be the unbounded
   component of $\C\setminus \beta$ that contains
   $\C\setminus\cl{V'}$. Then 
   \[ \wt{g} : W \to \C; 
         z\mapsto \frac{1}{2\pi i}\int_{\beta} \frac{f(\zeta)}{\zeta-z}d\zeta \]
   defines an analytic function on $W$. By the Cauchy integral theorem,
   $\wt{g}$ agrees with $g$ on $\C\setminus\cl{V'}$.
    
   Furthermore, for 
    $z\in W\cap V'$, we have by the residue theorem that
    \[ \wt{g}(z) - h(z) = \res_z\bigl(\frac{f(\zeta)}{\zeta-z}\bigr)
                     = f(z). \]
    In particular, $\wt{g}=g|_W$.
    Since $R$ was arbitrary, the claim follows. \end{subproof}

\begin{claim}[Claim 3]
   The function $h$ is uniformly bounded.
\end{claim}
 \begin{subproof}
    Let $z_0\in\C\setminus\alpha$. We set
    $\delta := \sin(\eta) = \dist(\alphat, \partial S)$ and
    define a curve
    $\gammat$ (depending on $z_0$) as follows.
    If $z_0\notin V$ or
    if $z_0\in V$ and $\dist(\Phi(z_0),\alphat)\geq \delta/2$, we
    simply set $\gammat := \alphat$. Otherwise, we set
     \[ \gammat := (\alphat \setminus 
                   \{z:|z-\Phi(z_0)|<\delta/2\}) \cup C, \]
    where $C$ is the arc of the circle $\{|z-\Phi(z_0)|=\delta/2\}$ for which
    $\alphat\cup C$ does not separate $\Phi(z_0)$ from $\infty$. 

   We also set $\gamma := \Phi^{-1}\circ\gammat$. 
    By Cauchy's integral theorem, we have
    \[ h(z_0) = \frac{1}{2\pi i}
             \int_{\alpha} \frac{\exp(\Phi(\zeta))}{\zeta-z}d\zeta =
       \frac{1}{2\pi i}\int_{\gamma} \frac{\exp(\Phi(\zeta))}{\zeta-z}d\zeta. \]
   Thus it is sufficient to show that the second integral is bounded
    independently of $z_0$. By the Koebe $1/4$-theorem and the definition
    of $\gamma$, we have
     $|\gamma(t)-z_0| \geq 
        \delta/8|\Phi'(\gamma(t))|$
    for all $t$. If we parametrize $\gammat$ by arclength, then clearly
     \[ \re\gammat(t)\leq C - K|t|, \]
    where the constants
    $K=|\re\nu|$ and $C$ are independent of $z_0$. We thus have
   \begin{align*}
    2\pi |h(z_0)| &= 
     \left|\int_{\gamma}
       \frac{\exp(\Phi(\zeta))}{\zeta-z_0}d\zeta\right| 
    \leq
       \int_{-\infty}^{+\infty} 
        \frac{|\exp(\gammat(t))|}{|\gamma(t)-z_0|}|\betat'(t)|dt  \\
     &=
       \int_{-\infty}^{+\infty}
        \frac{\exp(\re\gammat(t))}{|\Phi'(\gamma(t))|\cdot|\gamma(t)-z_0|}dt
     \leq \int_{-\infty}^{+\infty}
        \exp(C-K|t|)\frac{8|\Phi'(\gamma(t))|}{\delta|\Phi'(\gamma(t))|}
        dt \\
      &= \frac{8}{\delta} \int_{-\infty}^{+\infty}
          \exp(C-K|t|)dt <\infty. 
   \end{align*}
   So $h$ is uniformly bounded, as required. \end{subproof}

  To complete the proof,
   let $M>0$ such that $|h(z)| < M$ for all $z$. Set
   $K := 2M$ and
   $W := \{z\in\C: |g(z)|> K\}$.
   If $\beta$ is a simple closed 
   curve in $W$, then $|\Phi(z)|> M$ on $\beta$. By the minimum principle,
   we 
   also have $|\Phi(z)|>M$ on the region $U$ surrounded by
   $\beta$. It follows that $g$ has no zeros in $U$, and by the minimum
   principle $U\subset W$. Thus $W$ is simply connected.
 
   We can therefore define a function
    \[ G := \log g : W \to \{\zeta\in\C: \re \zeta > \log(2M)\}. \]
    It is easy to see that $G$ is proper. 
    Since there is exactly one homotopy class of
    curves in $W$ along which $G(z)\to \infty$, the degree of $G$
    is $1$. In other words, $G$ is a conformal isomorphism, and
    $f|_{W} = \exp\circ G$ is a universal cover, as required. 
   \end{proof}

 \begin{proof}[Proof of Proposition \ref{prop:approximation2}]
   Let $V := \exp(T)$, and let $\Psi:V\to \H$ be the conformal
    isomorphism with $\Psi\circ\exp = F$. Let $1<\rho<\min(\theta,2)$, 
    let $f$ be as in Proposition \ref{prop:approximation}, and let
    $\wt{g}$ be the entire function constructed there. Recall that this
    function satisfies
    $|\wt{g}(z)-f(z)|=O(1)$ on its tract 
    $W=\wt{g}^{-1}(\{|z|>K\})$. It easily follows that
    the logarithmic transform $\wt{G}$ also satisfies
    $|\wt{G}(z) - F(z)^{\rho}| \leq C_1$ for some 
    $C_1>0$. 

   Now set
    $g(z) := \wt{g}(z)/K$, and let 
    $G:\wt{T}\to\H$ be its logarithmic transform; i.e.\
    $G(z)=\wt{G}(z)-\log K$. We claim 
    that $g$ is the desired entire function. Indeed, by choice
    of $\wt{g}$, we have
     \[ |(F(z))^{\rho} - G(z)| \leq C_1 + \log K =: C, \]  
    which proves (\ref{item:sizeofG}). 

   To complete the proof, let $\gamma=\{z\in \wt{T}: |G(z)|=R\}$ be a
    vertical geodesic (where $R$ is sufficiently large; say $R\geq C+1$). 
    We need to prove that the diameter of $\gamma$
    is bounded independently of $R$. 
    So let $z\in\gamma$. 
    Then $|F(z)^{\rho} - G(z)|\leq C$, which implies that
    \[ \left| |F(z)| - R^{1/\rho}\right| \leq
       C\quad\text{and}\quad
       |\arg F(z)| < \pi/(1+\eps) \]
    (where $\rho=1+2\eps$), provided $R$ was chosen large enough. 

    The hyperbolic diameter of the subset of $\H$ described by these 
     inequalities --- and hence that of $F(\gamma)$ --- 
     is uniformly bounded.
     Since $F:T\to\H$ is a conformal isomorphism, the 
     standard estimate (\ref{eqn:standardestimate})
     on the hyperbolic metric on $T$, together with the fact that
     $T$ does not intersect its $2\pi i\Z$-translates,  implies 
     that the euclidean diameter of $\gamma$ is uniformly bounded 
     as well. 
  \end{proof}

 \section{Properties of the counterexample} \label{sec:properties}

 The goal of this section is to indicate how the counterexample 
  $f$ from Theorem \ref{thm:counterexample} (constructed in Section 
  \ref{sec:counter}) can be strengthened in various ways.
  We begin by discussing the growth behavior of the function $f$,
  and how to modify the construction to reduce this growth further.
  The section concludes with a sketch of the construction of a hyperbolic
  entire function whose Julia set contains no nontrivial curves at all.

 \subsection*{Order of growth}
 By Theorem \ref{thm:positive}, 
  we know that the counterexample $f$ from Theorem
   \ref{thm:counterexample} cannot have finite order; that is, we cannot have
   $\log \log |f(z)| = O(\log|z|)$. We now see that its growth is not
   all that much faster than this. 

\begin{prop}[Growth of counterexample] \label{prop:growth1}
 The function $f$ constructed in the proof of Theorem \ref{thm:counterexample} 
   satisfies
   \[ 
   \log \log |f(z)| = O\left((\log |z|)^{12M^2}\right) \,\,. 
   \]
\end{prop}
\begin{proof}
  We verify that the function 
    $F:T\to \H$ from Theorem \ref{Thm:RealizationTracts} satisfies
    \begin{equation} \log \re F(z) = O\left( (\re z)^{12M^2}\right).  \label{eqn:orderestimate}
    \end{equation}
   (The claim then follows immediately from the fact that $f$ is obtained
   from $F$ by applying Proposition \ref{prop:approximation2}).

  We use the notation of the proof of Theorem \ref{Thm:RealizationTracts}
   (recall Figure \ref{FigWiggleConstructionDetails}). 
   Pick points
   $p_k$ with real parts $\re(p_k) = \xi_k$ and satisfying
   $F(p_k)\in (\rho_{k+1},\infty)$
   (that is, $p_k$ lies half-way between the geodesics $C_k$ and $\sep{C}_k$,
   in the place where $T$ ``turns around'': it is here that the values of
   $\re F(z)$ are largest in terms of $\re z$).   
   It is not difficult to see that it is sufficient to verify
    (\ref{eqn:orderestimate}) when $z=p_k$. (In other parts of the 
    tract, $\log \re F(z)$ increases at most linearly with 
    $\re z$.) 

  We have 
   \[ \log \re F(p_k) \leq \log \sep{\rho}_{k+1} < 
              \log \sep \xi_{k+1}^{1/M} = 12\sep \xi_k. \]

  It remains to estimate $\sep{\xi}_k$ in terms of $\re(p_k) = \xi_k$, 
   which we can do because
   \begin{equation} \label{eqn:keyorderestimate}
    \sep\xi_k = \xi_k^{12M^2}
   \end{equation}
   by definition.
    So 
   \[ \log \re F(p_k) \leq 12\sep \xi_k = 12 \xi_k^{12M^2} =
            12 \re(p_k)^{12M^2}, \]
      as required. \end{proof}

 We are now going to discuss how to improve the growth behavior of $f$.
    Recall that $M>1$ was arbitrary; we will show how to reduce the 
    constant $12$ in the growth estimate to any number greater than $1$.
    Note that the main fact that influenced the growth of
    $f$ in the previous proof was (\ref{eqn:keyorderestimate}).
   We can improve the growth behavior of our counterexample by making the part
   of the tract leading up to $C_{k-1}$ thinner: 
   this will increase $R_k$ and hence $\xi_k$, while keeping
    $\sep{\xi}_k/\xi_k$ essentially the same. 

   More precisely, consider a tract described by a variation of Figure 
    \ref{FigWiggleConstructionDetails},
    where the upper of the three horizontal tubes connecting real parts
    $\xi_k$ and $\sep\xi_k$ has
    small height $\delta>0$, while the other two tubes remain at unit height. 
    In order for the proof to go through as before,
     $\xi_{k+1}$ will be roughly of size 
     $\exp(\sep\xi_k/(M\delta))$, while $\sep \xi_{k+1}$ should be chosen
     of size
      \[ \sep \xi_{k+1} \sim \xi_{k+1}^{M^2} \cdot \exp(CM\sep\xi_k). \]
    In other words, we will have
       \[ \sep \xi_{k+1} \lesssim \xi_{k+1}^{M^2(1+\delta C)}. \]

   (With such choices
    of $\xi_{k+1}$ and $\sep\xi_{k+1}$, better estimates on 
    $\rho_{k+1}$ and $\sep\rho_{k+1}$ are required in the 
    proof of 
    Theorem~\ref{Thm:RealizationTracts}. These are not difficult
    to furnish, but we shall not give the details here.)

     Hence we see, as in Proposition
     \ref{prop:growth1}, that
     \[ \log \re F(p_k) \leq C \re(p_k)^{M^2(1+\delta C)}. \] 
     By letting $\delta>0$ and $M>1$ be sufficiently small, 
     we have obtained the following result.

 \begin{prop}[Counterexamples of mild growth]
  For every $\eps>0$, there is a hyperbolic function $f\in\B$ such that
   $J(f)$ has no unbounded path-connected components, and such that
   \[ \log \log |f(z)| = O((\log |z|)^{1+\eps}).\qedoutsideproof \]
 \end{prop}

  Finally, we do not need to fix the height $\delta$, 
   but rather can let it tend to $0$ 
   in a controlled fashion, so that wiggles at large real parts have values of $\delta$ close to $0$.

Also note that, in all our examples, $\log \re F(z)$ grows at 
 most linearly with $\re(z)$ within the long
 horizontal tubes \emph{between} two consecutive ``wiggles'' (i.e., 
 between $\sep \xi_k+2h'+4$ and
 $\xi_{k+1}$ in Figure \ref{FigWiggleConstructionDetails}). 
 We claim that this means
 that the \emph{lower order} of $F$; i.e.\ the number
  \[ \liminf_{r\to\infty} \sup_{\re z=r} \frac{\log \re F(z)}{r} \]
 is finite. 

 Indeed, let $w_n$ and $\sep{w}_n$ be points at the beginning and
  the end of this tube, respectively. That is, $w_n$ is at real parts 
  slightly larger than $\sep \xi_k$ and $\sep{w}_n$ is at real parts slightly
  below $\xi_{k+1}$. We then have
    \[ |F(\sep{w}_n)| \leq |F(w_n)|\cdot \exp(C\cdot \re \sep{w}_n), \]
  where essentially $C=\pi$.
  It follows from the construction that 
   $|F(w_n)|$ grows at most like $\sep \xi_{k+1}$, and hence
   by (\ref{eqn:keyorderestimate})
   is bounded by $\xi_{k+1}^{A}$ for some $A>1$. So overall
    \[ \log \re F(\sep{w}_n) \leq A\log \xi_{k+1} + C\re \sep{w}_n \leq
                    A\log \re \sep{w}_n + C\re \sep{w}_n, \]
   and the lower order is at most $C$. 

 Since there are no other parts of the tract $T$ between the real parts of 
  $w_n$ and
  $\sep{w}_n$, 
  we can actually modify $T$ so that
  these tubes have the maximal possible height
  $2\pi$. Then the lower order of the resulting function $F$
  will be $C=1/2$, which is  the minimal possible
  value for a function in $\Blog$ by the Ahlfors
  distortion theorem \cite[Section 4.12]{ahlforsconformal}. 

 Altogether, this yields the following.

 \begin{prop}[More counterexamples of mild growth] 
   \label{prop:smallgrowthexample}
  There exists a function $F\in\Blog$ such that
  \begin{enumerate}
   \item $\displaystyle{\log \re F(z) = (\re z)^{1+o(1)}}$ as $\re z\to\infty$,
   \item $F$ has lower order $1/2$, and
   \item $J(F)$ has no unbounded path-connected components.
  \end{enumerate}
 \end{prop}

  Note that this function will not satisfy the stronger requirements in
   Theorem \ref{Thm:RealizationTracts} for a fixed $M$ (we need to
   let $M$ tend to $1$ as $k\to\infty$). 
   So we will not be able to use Proposition
   \ref{prop:approximation2} to obtain an entire function from $F$. (Also, 
   an application of Proposition \ref{prop:approximation2} 
   would slightly increase
   the lower order.) 
   We believe that it should be possible to modify 
   Proposition~\ref{prop:approximation2}
   so as to construct an entire function of class $\B$ with these
   properties.

\subsection*{No nontrivial path components}
 To conclude, we would like to note that our construction can also be
  adapted to yield a topologically
  stronger form of the counterexample. We content ourselves
  with giving a sketch of the proof, which involves a non-trivial amount of
  bookkeeping but is not conceptionally more involved than the previous
  arguments. 

 \begin{thm}[No Nontrivial Paths in the Julia Set] 
   \label{thm:pointcomponents}
  There exists a (hyperbolic) function $f\in\B$ such that
   $J(f)\cup\{\infty\}$ is a compact connected set that contains
   no nontrivial curve. 
 \end{thm}
 \sketch 
  Again, the result will be 
  established by designing a function $F\in\Blog$ with a single tract
  $T\subset \H$
  whose Julia set contains no nontrivial curve; the existence of an
  entire function with the same property is easily obtained using
  Proposition \ref{prop:approximation2}. 
  (Recall that $J(f)\cup\{\infty\}$
  is always a compact connected set when $f\in\B$, so only the
  second part of the claim needs to be established.)
 
  Let us say that a tract $T$ has a \emph{wiggle over $(r,R)$} if any curve
  in $T$ that connects a point at real part $r/2$ to 
  one at real part at least $2R$ contains at least three disjoint
  subcurves connecting the real parts $r$ and $R$. 

 \begin{figure}
\begin{center}
 \includegraphics[width=\textwidth]{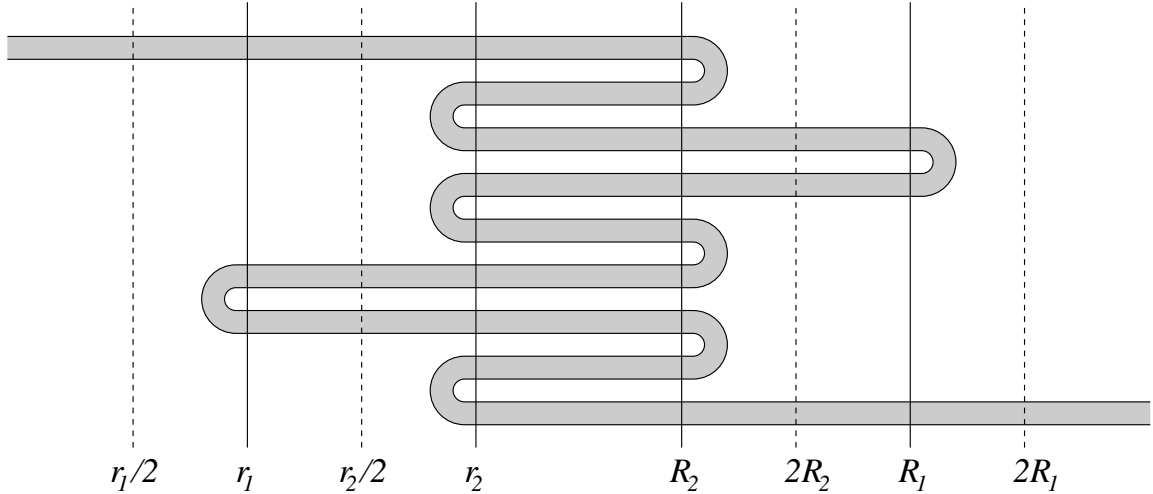}
\end{center}
  \caption{Illustration of the proof of Theorem
    \ref{thm:pointcomponents}. 
    The tract pictured here has a wiggle over
    $(r_1,R_1)$ and over $(r_2,R_2)$.\label{fig:manywiggles}}
 \end{figure}

 Our aim is now to construct a tract $T$, a conformal map $F:T\to\H$, and
  an associated set $\Wiggle$ of wiggles $(r,R)$ such that
  \begin{itemize}
   \item[1.] $T$ has a wiggle over $(r,R)$ for every $(r,R)\in\Wiggle$.
   \item[2.] For every $\eta\in [1,\infty)$, there is some wiggle 
     $(r,R)\in \Wiggle$ 
     with $\eta\leq r \leq R \leq 3\eta$. 
   \item[3.] 
    Every wiggle $(r,R)\in\Wiggle$ ``propagates'', roughly in the sense
     that curves connecting real parts $r$ and $R$ are going to map
     to an ``image wiggle'' $(r',R')\in\Wiggle$. 

 More precisely,
     suppose that $\gamma:[0,1]\to T$ connects real parts $r/2$ and $2R$,
     and $\re(\gamma(t))\in(r/2,2R)$ for all $t\in(0,1)$. Let us suppose
     without loss of generality that $|F(\gamma(0))|<|F(\gamma(1))|$. Then
     there should be $(r',R')\in \Wiggle$ such that, for every $t\in(0,1)$ with $\re \gamma(t)\in (r,R)$, we have
     \[ |F(\gamma(0))|< r'/2 < r' < |F(\gamma(t))| < R' < 2R' <
        |F(\gamma(1))|. \]
  \end{itemize}
  By linear separation of
    real parts (Lemma \ref{lem:realseparation}), for any two points
    $z,w\in J(F)$
    with the same external 
    address, there is 
    an iterate $F^k$ so that 
    $\re F^k(z)/\re F^k(w) > 12$ (assuming without loss of generality that
    $\re F^k(z)>\re F^k(w)$).
    So, by 2., there is a wiggle $(r,R)\in\Wiggle$ such that
     \[ \re F^k(z) < r/2 < 2R < \re F^k(w). \]
   The condition in 3.\ will then guarantee,
    by an inductive argument as in Theorem \ref{Thm:NoCurveToInfinity},
    that any curve in $J(F)$ connecting $F^k(z)$ and $F^k(w)$ would need to
    connect real parts $r$ and $R$ at least $3^n$ times for every $n$, 
    which is impossible.

 To complete our sketch, 
  we now indicate how to construct such a tract $T$, which will be
  a winding 
  strip contained in $\{|\im z|<\pi\}$, similarly as before. 
  However, the number
  of times that $T$ crosses the line $\{\re z = R\}$ will tend to
  infinity as $R$ does, so the width of $T$ will necessarily
  tend to $0$
  as real parts increase. Similarly as in Theorem
  \ref{Thm:RealizationTracts}, the tract will be constructed by inductively
  defining pieces $T_1, T_2, \dots$, in the following fashion:
  \begin{enumerate}
   \item $T_j$ is the piece of $T$ between real parts
     $\eta_{j-1}$ and $\eta_j$, where $\eta_0<\eta_1<\eta_2<\dots$ is
     a sequence tending to infinity.
   \item At each step in the construction, there is a set
     $\Wiggle_k = \{(r^k_1,R^k_1),\dots,(r^k_{m_k},R^k_{m_k})\}$
     of ``wiggles'', with
     $r^k_j/2 \geq \eta_{k-1}$ and $2 R^k_j \leq \eta_{k}$. 
     $T_k$ is constructed to have a wiggle over each
     $(r,R)\in\Wiggle_k$ (see Figure \ref{fig:manywiggles}).
   \item The next set of wiggles $\Wiggle_{k+1}$ is determined by the construction
     of $T_{k}$.
  \end{enumerate}

  More precisely, we begin by setting $\eta_0 := 1$,
    $\Wiggle_1 := \{(r_1,r_1+A)\}$, where $A$ is
    a sufficiently large number (fixed for the whole construction),
    and $r_1$ is large enough. We also set $\eta_1 := 2(r_1+A)>3\eta_0$.

   Given $\Wiggle_k$, we construct a piece $T_k$, connecting
    real parts $\eta_{k-1}$ and $\eta_k$, by first constructing
    a ``central curve'' that has a wiggle over every $(r,R)\in \Wiggle_k$
    (this is easy to achieve, compare Figure
     \ref{fig:manywiggles}),
    and then thickening this curve slightly (see below) to obtain $T_k$. 

     We then construct the set $\Wiggle_{k+1}$ as follows.
     Suppose that $(r,R)\in\Wiggle_k$, and that $\gamma:[0,1]\to T_k$ 
     is a minimal piece of
     the central curve of $T_k$ that connects real parts
     $r/2$ and $2R$. (Note that there may be several such pieces; we will add
     a wiggle to $\Wiggle_{k+1}$ for each of them.)

    Let $z$ be the first point on $\gamma$ that has
     real part $r$, and let $Z$ be the last point on $\gamma$ that has
     real part $R$.  Using the
     semi-hyperbolic metric, i.e.\ the reciprocal of the distance to 
     $\partial T_k$, we can estimate $|F(z)|$ and
     $|F(Z)|$ (up to an exponent of $2$), independently of
     the construction of $T_{K}$ for $K >k$. Hence we can add
     a new wiggle $(r^{k+1}_j,R^{k+1}_j)$ to $\Wiggle_{k+1}$ such that 
     $|F(z)|\geq r^{k+1}_j+A$ and
     $|F(Z)|\leq R^{k+1}_j-A$. 

    If the width of $T_j$ along $\gamma$ was chosen sufficiently
     thin, we can easily ensure that
     $r^{k+1}_j/2 > |F(\gamma(0))| > \eta_k$, and that
     $2R^{k+1}_j < |F(\gamma(1))|$. 

    Having added these finitely many wiggles to
     $\Wiggle_{k+1}$, we set $\eta_{k+1} := \max_{(r,R)\in \Wiggle_{k+1}} 2R$.
     Finally, we add sufficiently many wiggles 
     of the form $(t,t+A)$ to $\Wiggle_{k+1}$ to ensure that, for every 
     $\eta\in [\eta_{k},\eta_{k+1}/3]$, there is some wiggle between 
     real parts $\eta$ and $3\eta$. This completes the description of
     the inductive construction. \end{proof}

\appendix

\section{Some Geometric and Topological Facts}

 In this section, we collect some of the simple
  geometrical and topological results that we
  required in the course of the article. The first is a version of the
  \emph{Ahlfors spiral theorem} \cite[Theorem 8.21]{haymansubharmonic2}
  (which
  states that any entire function of finite order has
  controlled spiralling). We give a simple proof of this
  fact for functions in class $\Blog$ below. 
  In Section \ref{sec:growth}, we also required a characterization of
  domains with bounded wiggling, which we prove here for completeness.  
  Lemma \ref{Lem:GeometryGeodesic} below was used in 
  Theorem \ref{Thm:RealizationTracts}. 

 Finally, the Boundary Bumping Theorem \ref{thm:boundarybumping} 
  was used a number of times in topological considerations, and
  Theorem \ref{thm:arccharacterization} was instrumental in the 
  proof of Theorem \ref{thm:positive}.

 \begin{thm}[Spiral Theorem] \label{thm:spiral}
  Suppose that $F\in\Blog$ has finite order. Then the tracts
   of $F$ have bounded slope.
 \end{thm}
 \begin{proof} Let $T$ be a tract of $F$,
   set $\rho := \sup\{\frac{\log \re F(z)}{\re z}:z\in\H\cap\T\}<\infty$,
   and consider the central geodesic 
   $\gamma:[1,\infty)\to T; t \mapsto \FInv{T}(t)$.
   Then for every $t\geq 1$,
\[ 
	|\gamma(t)| - |\gamma(1)| \leq |\gamma(t)-\gamma(1)| \leq
       2\pi      \ell_{T}\bigl(\gamma\bigl([1,t]\bigr)\bigr) =
           2\pi\log t \leq 2\pi \rho\re \gamma(t)\;.
 \]
Thus we have proved the existence of an asymptotic curve $\gamma$
satisfying
$|\im\gamma(t)|\leq |\gamma(t)|\leq K\re\gamma(t)+M$, for $K=2\pi\rho$ and $M=|\gamma(1)|$, which is equivalent to the bounded slope condition. 
\end{proof}

\begin{lem}[Domains with bounded wiggling]
\label{lem:boundedwiggling}
 Let $V$ be an unbounded Jordan domain such that
  $\exp|_V$ is injective.
  Suppose that there are 
  $K,M>0$ such that every $z_0\in V$ can be connected to $\infty$ by
  a curve $\gamma\subset V$ satisfying 
   \[
     \re z \geq \frac{\re z_0}{K} - M \]
  for all $z\in\gamma$. Then there is $M'>0$ that depends only
  on $M$  such that, for every $z_0\in V$,
   \[
     \re z \geq \frac{\re z_0}{K} - M' \]
   for all $z$ on the
   geodesic connecting $z_0$ to $\infty$.    
\end{lem}
\begin{proof}
  Let $z_0\in V$, let $\gamma$ be a curve as in the statement
  of the theorem, and let
  $F:V\to\H$ be a conformal isomorphism with $F(\infty)=\infty$ and
  $F(z_0)=1$. Then
   \[ \gamma' := F^{-1}\bigl([1,\infty)\bigr) \]
  is the horizontal geodesic connecting $z_0$ to $\infty$. 

  Let $z\in \gamma'$. 
   By \cite[Corollary 4.18]{pommerenke}, we can find geodesics
   $\alpha^+$ and $\alpha^-$ of $\H$,
   connecting $F(z)$ to the positive resp.\
   negative imaginary axis, such that the geodesics
   $F^{-1}(\alpha^{\pm})$ of $V$ have diameter at most
   $C\dist(z,\partial V)$. (Here $C$ is a universal constant.) 
   Hence the crosscut $\alpha := F^{-1}(\alpha^+)\cup F^{-1}(\alpha^-)$, 
   which separates $z_0$ from $\infty$ in $V$, has diameter 
   at most $2C\dist(z,\partial V) \leq 4C\pi$. 

  The curve $\gamma$ must intersect $\alpha$ in some point
   $w$. We thus have
   \[ \re z \geq \re w - 4C\pi \geq \re z_0/K - M - 4C\pi\;.
      \qedhere \]
\end{proof}

\begin{figure}
  \resizebox{.8\textwidth}{!}{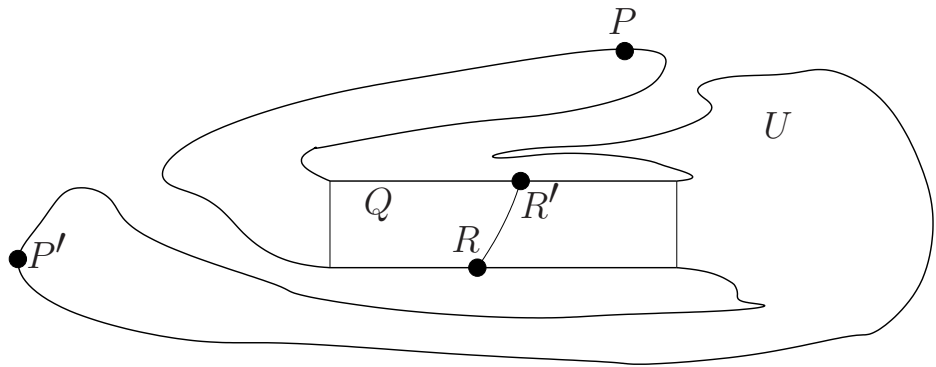}
 \caption{Illustration of Lemma \ref{Lem:GeometryGeodesic}.}
\end{figure}

\begin{lem}[Geometry of Geodesics]
\label{Lem:GeometryGeodesic}
Consider the rectangle $Q=\{z\in\C\colon |\re z|<4, |\im z|<1\}$ and let
 $U\subset \Ch$ be a simply connected Jordan domain with 
 $U\supsetneq Q$ such 
 that $\partial Q\cap\partial U$ consists  exactly of
 the two horizontal boundary sides of $Q$. Let $P,R,P',R'\in\partial U$ be 
 four distinct boundary points in this cyclic order, 
 subject to the condition that $P$ and $P'$ are in the boundary of different 
 components of $U\sm Q$, and so that the quadrilateral $U$ with the marked 
 points $P,R,P', R'$ has modulus $1$.

Let $\gamma$ be the hyperbolic geodesic in $U$ connecting $R$ with $R'$. If 
 $0\in\gamma$, then the two endpoints of $\gamma$ are on the horizontal 
   boundaries of $Q$, one endpoint each on the upper and lower boundary.
\end{lem}
\begin{remark}
 This is essentially a simple version of the well-known
  \emph{Ahlfors distortion theorem} (see e.g.\ 
  \cite[Section 4.12]{ahlforsconformal} or 
  \cite[Section 11.5]{pommerenke}). However, in the way it is usually
  stated, this theorem 
  cannot be applied directly to obtain our lemma. Hence we provide
  the proof for completeness, following 
  \cite[Section 4.12]{ahlforsconformal}.
\end{remark}
\begin{proof}
 We need to show that $\gamma$ 
 does not cross the left side $L$ or the right side $R$ of the rectangle $Q$. 
 We show this for the left side; the statement for the right side follows by 
 symmetry. 
 Let $M$ be the vertical crosscut of $Q$ that passes through $0$, 
 and let $Q'$ be the square bounded by $L$, $M$ and the horizontal boundaries 
 of $Q$. (That is, $Q'$ is the ``left half'' of $Q$.)

  Let $\phi$ be the conformal map that takes $U$ to the bi-infinite strip 
  $\{0<\im z < \pi\}$ in such a way that $P$ and $P'$ are mapped to 
  $-\infty$ and $+\infty$ and $R$ and $R'$ are mapped to $\pi i$ and $0$. Set
    \[ l := \sup_{z\in L} \re \phi(z)
    \quad\text{and}\quad m := \inf_{z\in M} \re \phi(z) \leq 0; \]
  we need to show that $l \leq 0$.

 Let $\wt{Q'}$ be the quadrilateral obtained by joining $\phi(Q')$ with its
  reflection in 
 the real axis; then $\wt{Q'}$ has modulus equal to $1/2$. The exponential map 
 takes $\wt{Q'}$ together with its upper and lower boundaries to an annulus 
 surrounding $0$ that also has modulus $1/2$. 

 The following is a consequence of the 
  modulus theorem of Teichmueller (see e.g.\
  \cite[Theorem 4.7]{ahlforsconformal} and the subsequent remark):
 \emph{If the annulus $A$ separates the points $0$ and $z$ from the points
   $\infty$ and $w$, where $|w|<|z|$, then $\mod(A)<1/2$.}

 This implies that $l\leq m$, as required.
\end{proof}

We conclude by stating two results of continuum theory that are used in this
 article.

\begin{thm}[Boundary Bumping theorem {\cite[Theorem 5.6]{continuumtheory}}]
  \label{thm:boundarybumping}
 Let $X$ be a nonempty, compact, connected metric space, and let
  $E\subsetneq X$ be nonempty. If $C$ is a connected component of $E$, then
  $\partial C \cap \partial E\neq\emptyset$ (where boundaries are taken
  relative to $X$). 
\end{thm}
 \begin{remark}
   We apply this theorem only in the case where $X\subset\Ch$ is a 
    compact connected set containing $\infty$, and $E=X\cap\C$. In this case,
    the theorem states that every component of $E$ is unbounded. 
 \end{remark}

\begin{thm}[Order characterization of an arc {\cite[Theorems 6.16 \& 6.17]{continuumtheory}}]
  \label{thm:arccharacterization}
 Let $X$ be a nonempty, compact, connected metric space. Suppose that
  there is a total ordering $\prec$ on $X$ such that the order topology
  of $(X,\prec)$ agrees with the metric topology of $X$. Then 
  either $X$ consists of a single point or
  there is an order-preserving homeomorphism from $X$ to the unit interval
  $[0,1]$.
\end{thm}
\begin{remark}
 This result follows from the perhaps
  better-known \emph{non-cut-point characterization}
  of the arc: a compact, connected metric space 
  is homeomorphic to an arc if and only if it has
  exactly two non-cut-points. Conversely, this characterization also
  follows from Theorem \ref{thm:arccharacterization}; 
  see \cite[Theorem 6.16]{continuumtheory}
  for details. 
\end{remark}

\providecommand{\href}[2]{#2}\def\polhk#1{\setbox0=\hbox{#1}{\ooalign{\hidewid%
th \lower1.5ex\hbox{`}\hidewidth\crcr\unhbox0}}}
\input{cyracc.def} 
  \newfont{\cyrit}{wncyi10 at 12pt}\def\cprime{$'$}
\providecommand{\bysame}{\leavevmode\hbox to3em{\hrulefill}\thinspace}



\end{document}

%% file: symboldefs.tex
\newcommand{\N}{\mathbb{N}}

\newcommand{\Z}{\mathbb{Z}}

\newcommand{\R}{\mathbb{R}}

\newcommand{\C}{\ensuremath{\mathbb{C}}}
\newcommand{\Ch}{\hat{\mathbb{C}}}

\newcommand{\eps}{\ensuremath{\varepsilon}}

\newcommand{\id}{\operatorname{id}}

\newcommand{\re}{\operatorname{Re}}
\newcommand{\im}{\operatorname{Im}}

\newcommand{\cl}[1]{\overline{#1}}
\newcommand{\dist}{\operatorname{dist}}

\newcommand{\B}{\mathcal{B}}

\newcommand{\Blog}{\B_{\log}}
\newcommand{\Hlog}{\mathcal{H}_{\log}}

\newcommand{\F}{\mathcal{F}}

\newcommand{\sing}{\operatorname{sing}}

\newcommand{\wt}[1]{\widetilde{#1}}

\newcommand{\T}{\mathcal{T}}
\renewcommand{\H}{\mathbb{H}}
\renewcommand{\phi}{\varphi}

\newcommand{\gammat}{\tilde{\gamma}}

\newcommand{\alphat}{\tilde{\alpha}}
\newcommand{\betat}{\tilde{\beta}}

\newcommand{\res}{\operatorname{res}}

\renewcommand{\rho}{\varrho}

\newcommand{\V}{\mathcal{V}}

\newcommand{\s}{\underline{s}}

\newcommand{\sm}{\setminus}

\newcommand{\FInv}[1]{F^{-1}_{#1}}

\newcommand{\wh}[1]{\widehat{#1}}

%% file: standard.tex
\usepackage{graphicx,color,subfigure,ifthen}
\usepackage{version}
\usepackage[totalwidth=155mm,totalheight=223mm,centering]{geometry}

\usepackage{amsmath}
\usepackage{amssymb}

\makeatletter
\def\swappedhead#1#2#3{%
  \thmnumber{\@upn{\the\thm@headfont #2\@ifnotempty{#1}{.~}}}%
  \thmname{#1}%
  \thmnote{ {\the\thm@notefont(#3)}}}
\makeatother

 \newtheoremstyle{changebreak}
   {9pt}
   {9pt}
   {\itshape}
   {}
   {\bfseries}
   {.}
   {\newline}
   {}

 \newtheoremstyle{claimstyle}%
   {}
   {}
   {\normalfont}
   {}
   {\itshape}
   {.}
   { }
   {\thmnote{#3}}

\theoremstyle{claimstyle}
\newtheorem*{varclaim}{}

\newenvironment{claim}[1][Claim]{\begin{varclaim}[#1]}{\end{varclaim}}
\newenvironment{remark}[1][Remark]{\begin{varclaim}[#1]}{\end{varclaim}}

\newcommand{\proofsymbol}{\mbox{$\blacksquare$}}
\newcommand{\subproofsymbol}{\mbox{$\triangle$}}

\newenvironment{subproof}{\begin{proof}%
               }{%
             \end{proof}}
\newenvironment{sketch}{\begin{proof}[Sketch of proof]}{%
               \end{proof}}

\newcommand{\qedoutsideproof}{%
    \pushQED{\qed}\qedhere}

\theoremstyle{changebreak}
\swapnumbers

\newtheorem{thm}{Theorem}[section]
\newtheorem{defn}[thm]{Definition}
\newtheorem{lem}[thm]{Lemma}
\newtheorem{cor}[thm]{Corollary}

\newtheorem{prop}[thm]{Proposition}

\newtheorem{deflem}[thm]{Definition and Lemma}

\swapnumbers

\renewcommand{\mod}{\operatorname{mod}\,}